\documentclass[9pt,twoside,english,reqno,a4paper]{amsart}

\usepackage{listings,graphicx,amsmath,varioref,amscd,amssymb,color,bm,stmaryrd,amsthm,amsfonts,graphics,geometry,latexsym,pgf,pst-all} 
\theoremstyle{plain}
\usepackage{esint}
\usepackage{amsthm}

\usepackage{enumerate}

\theoremstyle{plain}
\newtheorem{theorem}{Theorem}[section]
\newtheorem{proposition}[theorem]{Proposition}
\newtheorem{Proposition}[theorem]{Proposition}
\newtheorem{lemma}[theorem]{Lemma}

\usepackage{geometry}
\usepackage{cancel}
\usepackage[colorlinks=false]{hyperref}

\theoremstyle{definition}
\newtheorem{defin}[theorem]{Definition}

\newtheorem{remark}[theorem]{Remark}
\newtheorem{example}{Example}

\theoremstyle{remark}

\usepackage{fouriernc}
\usepackage[T1]{fontenc}

\font\manual=manfnt
\newcommand\xqed[1]{%
  \leavevmode\unskip\penalty9999 \hbox{}\nobreak\hfill
  \quad\hbox{#1}}
\newcommand\triang{\xqed{\manual\char'170}}

\def\bk{\color{black}}

\numberwithin{equation}{section}

\def\dis{\displaystyle}

\DeclareMathOperator{\R}{\mathbb{R}}
\DeclareMathOperator{\N}{\mathbb{N}}

\newcommand{\car}[1]{\raise1pt\hbox{$\chi$}_{#1}}

\newcommand{\DM }{\mathcal{DM}^\infty }
\def\rn{\mathbb{R}^N}

\def\re{\mathbb{R}}
\newcommand{\res}{\!\!\mathop{\hbox{
			\vrule height 7pt width .5pt depth 0pt
			\vrule height .5pt width 6pt depth 0pt}}
	\nolimits}

\begin{document}
\title[Global $BV$ regularity for 1-Laplace type problems with singular lower order terms]{Optimal global $BV$ regularity for 1-Laplace type BVP's with singular lower order terms}

\author[A. J. Mart\'inez Aparicio]{Antonio J. Mart\'inez Aparicio}
\author[F. Oliva]{Francescantonio Oliva}
\author[F. Petitta]{Francesco Petitta}

\address[Antonio J. Mart\'inez Aparicio]{Departamento de Matem\'aticas,
Universidad de Almer\'ia
	\hfill \break\indent
    Ctra. Sacramento s/n, La Ca\^{n}ada de San Urbano, 04120 Almer\'ia, Spain}
\email{\tt ajmaparicio@ual.es}

\address[Francescantonio Oliva]{
Dipartimento di Scienze di Base e Applicate per l' Ingegneria, Sapienza Universit\`a di Roma
	\hfill \break\indent
	Via Scarpa 16, 00161 Roma, Italy}
\email{\tt francescantonio.oliva@uniroma1.it}

\address[Francesco Petitta ]{Dipartimento di Scienze di Base e Applicate per l' Ingegneria, Sapienza Universit\`a di Roma
		\hfill \break\indent
		Via Scarpa 16, 00161 Roma, Italy}
\email{\tt francesco.petitta@uniroma1.it}

\makeatletter
\@namedef{subjclassname@2020}{%
  \textup{2020} Mathematics Subject Classification}
\makeatother

\begin{abstract}
 In this paper we  provide a complete characterization  of the regularity properties of the solutions associated to the homogeneous Dirichlet problem
\begin{equation*}
    \begin{cases}
    \displaystyle - \Delta_1 u= h(u)f &  \text{in } \Omega, \\
    u=0 & \text{on } \partial \Omega,		
    \end{cases}
\end{equation*}
where $\Omega\subset\mathbb{R}^N$ is a bounded open set with Lipschitz boundary,  $f \in L^m(\Omega)$ with $m\geq 1$ is a nonnegative function and $h\colon \mathbb{R}^+ \to \mathbb{R}^+$ is continuous, possibly singular at the origin and bounded at infinity.

 Without any growth restrictions on $h$ at zero, we prove existence of global finite energy solutions in $BV(\Omega)$ under sharp conditions on the summability of $f$ and on the behaviour of $h$ at infinity. Roughly speaking, the faster $h$ goes to zero at infinity, the less regularity is required on $f$. In contrast to the $p$-Laplacian case ($p>1$),  we show that the behaviour of $h$ at the origin plays essentially no role.

The main result contains an extension of the celebrated one of Lazer-McKenna (\cite{lm}) to the case of the $1$-Laplacian as principal operator. 
\end{abstract}


\maketitle
\tableofcontents

\section{Introduction}

Over the last few decades,  many mathematicians have focused their attention to singular boundary value problems, whose simplest model is given by 
\begin{equation}
	\begin{cases}
 - \Delta u= f(x){u^{-\gamma}} \ &\text{in}\, \ \Omega,
 		\\
 		u\ge 0 \  &\text{in}\,\ \Omega,
		\\
 		u=0 \  &\text{on}\,\ \partial\Omega,
\label{pbintro}
	\end{cases}
\end{equation}
where $\Omega\subset\mathbb{R}^N (N>1)$ is a smooth bounded  domain, $\gamma>0$ and $f$ is a nonnegative function.   In this context, two issues have been intensely addressed in particular\bk. The first, for obvious reasons, is the existence of solution. The second, also due to  its interest in physics, is whether such solutions have finite energy. When dealing with the Laplacian operator, the finite energy space is  $H_0^1(\Omega)$. \bk
\medskip

 Regarding the existence of a solution, when $f\in C^{\alpha}(\overline{\Omega})$ (for some $0<\alpha<1$)  is bounded away from zero on $\overline{\Omega}$, the authors in~\cite{crt} proved both existence and uniqueness of a classical solution to~\eqref{pbintro}. Then, in the seminal paper~\cite{lm},  A.C. Lazer and P.J. McKenna showed that $u\notin C^{1}(\overline{\Omega})$ if $\gamma> 1$ while $u\in H^{1}_0(\Omega)$ if and only if $\gamma<3$.
\medskip

In the general context of distributional solutions  a pioneering work was~\cite{bo}, where the authors studied the existence of solutions to~\eqref{pbintro} for a nonnegative $f$ just belonging to $L^{m}(\Omega)$. \bk In particular, such solutions are shown to satisfy the following regularity properties: 
 	\begin{itemize}		
		\item[i)] if $\gamma<1$ and $m= \left(\frac{2^*}{1-\gamma}\right)'$ then $u \in H^1_0(\Omega)$, otherwise if $1 \le m <\left(\frac{2^*}{1-\gamma}\right)'$ then $u \in W^{1,\frac{Nm(\gamma+1)}{N-m(1-\gamma)}}_0(\Omega)$;
		\item[ii)] if $\gamma=1$  and $m = 1$ \bk then $u \in H^1_0(\Omega)$;
		\item[iii)] if $\gamma>1$  and $m = 1$ \bk then $u \in H^1_{\rm{loc}}(\Omega)$.
	\end{itemize}
	
In the latter case (i.e. \text{iii)}), the boundary datum is meant in a weak sense, namely
\begin{equation}\label{potenza}
u^{\frac{\gamma+1}{2}} 	\in H^1_0(\Omega).	
\end{equation}
Let us underline that the previous request is not only technical as shown in Example~\ref{ex1} below; in fact,  for $\gamma>1$, \bk there exist distributional solutions to \eqref{pbintro} not belonging to $H^1_0(\Omega)$ but satisfying~\eqref{potenza}. We also mention that a different weak way to infer the homogeneous boundary datum for a nonnegative function $u$ consists in requiring $(u-k)^+=0$ on $\partial \Omega$ for any $k>0$ (see \cite{cd,cs}) provided  $(u-k)^+$ has a trace. 
\medskip

In this frame, the Lazer-McKenna threshold for the existence of finite energy solutions when $\gamma>1$ was extended to the case of a general $f\in L^{m}(\Omega)$: problem~\eqref{pbintro} has a solution (which is unique) in $H_0^1(\Omega)$ for any  positive \bk $f\in L^{m}(\Omega)$ ($m> 1$) if and only if $\gamma<{3-\frac{2}{m}}$ (see \cite{OP}).
\medskip

Some \bk  of the results discussed above have been suitably extended  to the case of nonlinear operators. In order to be concrete, let us consider the following model problem: 
\begin{equation} \label{op}
    \begin{cases}
    \displaystyle - \Delta_p u= f(x)u^{-\gamma} &  \text{in}\, \ \Omega, \\
    u\geq 0& \text{in}\, \ \Omega,\\
    u=0 & \text{on}\, \ \partial \Omega\,,		
    \end{cases}
\end{equation}

where  $\Delta_p u = \operatorname{div}(|\nabla u|^{p-2} \nabla u)$ is the usual $p$-Laplace operator for $p>1$. 
\medskip
 
One could wonder when problem \eqref{op} admits solutions $u$ with global finite energy, i.e. whether \bk $u\in W^{1,p}_0(\Omega)$. In this case, at least for smooth positive data, the Lazer-McKenna threshold in order to get finite energy solutions becomes (\cite{S}) 
\begin{equation}\label{condLMp}
 \gamma<\frac{2p-1}{p-1}.
 \end{equation}
 
More in general, if $f$ is an $L^m(\Omega)$ function with enough summability, in \cite{DCA} (see also \cite{DDO}) was shown the existence of distributional solutions $u\in W^{1,p}_{\rm loc}(\Omega)$ such that 
\begin{equation} \label{formally}
    u^{ \max\left(1,\frac{\gamma-1+p}{p}\right)}\in W^{1,p}_0(\Omega),
\end{equation}
so that, if $\gamma>1$, coherently with \cite{bo} in case $p=2$, only a power of the solution is shown to have finite energy.  Again, as for the Laplacian case, we stress that when $\gamma>1$ the solution to~\eqref{op} may not belong to the finite energy space $W_0^{1,p}(\Omega)$ (see Example~\ref{ex1}). For a more in-depth discussion of singular problems of this type, we refer the reader to the recent survey \cite{opSU}.
\medskip

Now let us focus our attention on problems as \eqref{op} with $p=1$, where the operator becomes $-\Delta_1 u = -\operatorname{div}(|D u|^{-1}D u)$. Here, the finite energy space turns out to be $BV(\Omega)$, the space of functions with bounded total variation over $\Omega$. A major difficulty arises in giving a meaning to the quotient $|D u|^{-1}D u$, since $Du$ is merely a measure. To address this, a suitable notion of solution was given in~\cite{ABCM}, where the authors introduce a vector field $z\in L^\infty(\Omega)^N$ that plays the role of this quotient. This vector field is linked to $u$ through the relation $(z,Du)=|Du|$, where the generalized dot product $(z,Du)$ is defined via the $L^\infty$-divergence-measure vector fields theory (\cite{An}, \cite{CF}). In \cite{DGOP}, approximating through solutions $u_p$ to \eqref{op} with $p>1$, it is shown the existence and (provided $f>0$) uniqueness  of nonnegative bounded solutions to
\begin{equation}\label{now}
    \begin{cases}
        \displaystyle - \Delta_1 u= f(x){u^{-\gamma}} &  \text{in} \ \Omega, \\
 		u=0 & \text{on}\ \partial \Omega,
 		\,
    \end{cases}
\end{equation}
where $0\le f\in L^{N}(\Omega)$. Let us stress that the summability exponent $N$  represents the well known "Pillars of Hercules" for problems arising from Lagrangian with linear growth as 
$$
F(v)=\int_\Omega |D  v|-\int_\Omega f v \, dx,  
$$
(see for instance \cite{ACM,ct,mst,OCP} and references therein); beyond this threshold substantial degenerate and singular phenomena appear (see   \cite{MST2009}) unless some regularizing lower order terms are considered (see for instance \cite{o,lops,bop} \bk and the discussion in  Section \ref{sec6} below). 
\medskip
 
Let us precise that, as for the case $p>1$ with sufficiently integrable data, if $\gamma\leq 1$ (see \cite{DGS}) and $f\in L^{N}(\Omega)$ is nonnegative, then finite energy solutions to problem \eqref{now} always exist. \bk
\medskip

 At first glance,  if $\gamma>1$ the situation seems to be in continuity to the one of the $p$-Laplacian in~\eqref{formally}; \bk in order to better understand this phenomenon  one can formally multiply the equation in \eqref{now} by $u^{\gamma}$ yielding to an estimate of the form 
$$
\int_{\Omega} |D u^{\gamma}| + \int_{\partial\Omega} u^{\gamma} \, d\mathcal{H}^{N-1}\leq \int f \, dx  ,
$$
namely $u^{\gamma}\in BV(\Omega)$. This was rigorously shown in \cite{DGOP} by sending $p\to 1^+$ into \eqref{op} and it is clearly consistent with \eqref{formally}, as for the case $p>1$. 
\medskip

 Asking for  $f\in L^N(\Omega)$  allows the authors to prove the existence of nonnegative bounded solutions to~\eqref{now}. Also, when $f$ merely belongs to $L^m(\Omega)$, $1\leq m<N$, a priori, solutions may be unbounded and, even more so, they  may not belong to  $BV(\Omega)$\bk.  Although, a key observation in this more general setting is that problem~\eqref{now} always has a solution $u$ such that $T_k(u)^{\max (1,\gamma)}\in BV(\Omega)$ for every $k>0$, where $T_k$ denotes the truncation function (see  \cite{lops}).

\medskip
Inspired by \cite{lm}, the  relevant question  addressed here is \bk under which  minimal conditions  solutions $u$ to \eqref{now} have global  finite energy, i.e. whether \bk  $u\in BV(\Omega)$. \bk  In order to have some insight, let us  come back for a moment to \eqref{condLMp}; as $p\to 1^+$ it suggests  that one  could expect to obtain a global  $BV$  solution to \bk\eqref{now} for any $\gamma>0$ at least for both a non-degenerate bounded  datum and a smooth domain.  
\medskip

Let us  also observe that this suggestion could be strengthen by the following argument: for $p>1$ infinite energy solutions can be constructed  due to the behaviour of the approximating solutions $u_p$  near the boundary of $\Omega$ as 
$$
u_p \sim d(x,\partial\Omega)^{\frac{p}{p+\gamma-1}}.
$$
Then, as $\gamma$ grows, around any boundary point of  $\Omega$, \bk the solutions $u_p$ of \eqref{op} become steepest and steepest until they cease to belong to $W^{1,p} (\Omega)$ as $\gamma$ reaches $\frac{2p-1}{p-1}$.  
Even if thought \bk as limit of such $u_p$'s,  the guesswork that the limit  solution $u$ of \eqref{now} should have finite energy can also be  fuelled by the fact that a $BV$ function allows jumps, for instance at the boundary, and the fact that  zero trace in \eqref{now}  needs not to be  attained in a classical sense of traces (see for instance \cite{ACM, DGOP}, and references therein). This is the case, for example, of constant solutions. 
\medskip

At least for $\Omega$ smooth, the idea of having nonnegative solutions to~\eqref{now} belonging to $BV(\Omega)$ for any $\gamma>0$ may be reinforced by looking at the model case 
\begin{equation}\label{lm}
    \begin{cases}
    \displaystyle - \Delta_1 u= {u^{-\gamma}} &  \text{in}\, \ \Omega, \\
    u=0 & \text{on}\, \partial \Omega.
    \end{cases}
\end{equation}

For this problem, global $BV$ regularity of nonnegative solutions when $\gamma>1$ was shown to hold \bk (see \cite[Example 1]{DGOP})
in a rich  class of domains, say $C^{1,1}$ domains,   by observing   that the unique solution of \eqref{lm} is constant and, in particular, it is a suitable power of the Cheeger constant $\lambda_{\Omega}$, namely $$u=\left(\frac{|\Omega|}{\mathrm{Per}(\Omega)}\right)^{\frac{1}{\gamma}}. $$ 

Nevertheless, for general  Lipschitz domains, the question becomes more delicate since it is possible to construct  non-constant solutions of \eqref{lm} whose regularity could  be, a priori, difficult to compute.  For instance, if $\Omega$ is convex, then 
$$
u=\left(-H_\Omega(x)\right)^{-\frac{1}{\gamma}}
$$
is the unique solution of \eqref{lm} (see Example \ref{example2} below),  where $H_\Omega(x)$ is the variational mean curvature of $\Omega$ (see \cite{BGT} for details). Without entering into technicalities, only recall that $H_\Omega (x)\leq 0$ for any $x\in \re^N$, that $H_\Omega$ agrees with the Cheeger constant $\lambda_{\Omega}$ of $\Omega$ on its (unique) Cheeger set, and that $\|H_\Omega\|_{L^{\infty}(\rn)}< \infty$ if and only if $\Omega$ is of class $C^{1,1}$; hence, $u=0$   at non-regular points of $\Omega$ (e.g. at corners). 
\medskip

In this paper, more in general,  we consider nonnegative solutions to 
\begin{equation}
	\label{eq:PbMainIn}
	\begin{cases}
		\dis -\Delta_1 u = h(u)f & \text{in}\;\Omega,\\
		u=0 & \text{on}\;\partial\Omega,
	\end{cases}
\end{equation}
where $\Omega\subset \R^N$ is a Lipschitz bounded domain, $0\leq f \in L^{m}(\Omega)$ with $m\geq 1$ and $h\colon \mathbb{R}^+\to\re^+$ is continuous, possibly singular at $s=0$, and bounded at infinity. Observe that uniqueness may fail to hold, in general,  for problems as \eqref{eq:PbMainIn}. Indeed, as we have seen, solutions can be constant and then $h$ can not have a flat zone in order to expect uniqueness. Nevertheless, there exists at most one finite energy solution if $h$ is decreasing and $f>0$ in $\Omega$ \bk as shown in~\cite[Theorem $3.5$]{DGOP} and~\cite[Theorem $3.4$]{lops}.
\medskip

To summarize, we aim to investigate the following question:
\medskip

\begin{center}{ \center \it When does problem  \eqref{eq:PbMainIn} have a solution in  $BV(\Omega)$?}\end{center}
\medskip

As we will see, the answer to this question depends both on the summability of $f$ and on the behaviour of $h$ at infinity. In broad terms, the price to pay for $f$ being less integrable is that $h$ goes faster to zero at infinity. Surprisingly, the singularity of $h(s)$ at $s=0$ plays no role in the regularity of the solution; for instance, problem \eqref{now} always has a $BV(\Omega)$ solution independently of the value of $\gamma$ under suitable summability assumptions. This is a striking difference with respect to singular $p$-Laplacian problems.
\medskip

In order to be concrete, we show that a $BV(\Omega)$ solution for problem \eqref{eq:PbMainIn} exists if one of the following cases occurs (see Theorem \ref{teo_ex_F}):
\begin{enumerate}[i)]
    \item $m= N$ and $\|f\|_{L^N(\Omega)} <  (\mathcal{S}_1 \displaystyle \limsup_{s\to\infty} h(s))^{-1}$, where $\mathcal{S}_1$ is the Sobolev constant (see Section \ref{sec_prel});
    \item $m=\left( \frac{1^*}{1-\theta} \right)'=\frac{N}{(N-1)\theta + 1}$ and $\displaystyle \limsup_{s\to\infty} h(s) s^\theta < \infty$ for some $0< \theta < 1$;
    \item $m=1$ and $\displaystyle \limsup_{s\to\infty} h(s) s < \infty$.
\end{enumerate}
\medskip

Moreover, all these conditions are shown to be optimal. Therefore, the behaviour of the nonlinearity at infinity produces a regularizing effect in the solutions of problem~\eqref{eq:PbMainIn}. Let us highlight that these results, presented in a wide generality, are new even for nonsingular nonlinearities. Observe also that the case $m<N$ is of particular interest since the solutions always have finite energy despite being in general unbounded.
\medskip

The basic idea to avoid the restrictions that the singularity of $h(s)$ at $s=0$ usually produces relies in  using \bk a suitable nonpositive test function whose derivative behaves as the derivative of the solution. This unexpectedly allows us to show that the behaviour of $h(s)$ when $s$ is small plays a role which is negligible. Observe that problem~\eqref{now} when $\gamma>1$ is included in the above third case and, as a consequence, a finite energy solution exists for {every Lipschitz domain $\Omega$} and every nonnegative $f\in L^1(\Omega)$.  Note that, as a by-product,   our result  extends the Lazer-McKenna threshold~\eqref{condLMp} (only known for smooth data) to the case $p=1$.\bk 

\medskip

In summary, this work represents a  significant  improvement  in the study of $1$-Laplacian type problems with very general nonlinearities. \bk On the one hand, it unifies the study of singular problems, showing that when dealing with the $1$-Laplacian there is no difference between what have traditionally been called mild ($\gamma\leq 1$) and strong ($\gamma>1$) singularities. On the other hand, it also highlights the effect that the behaviour of $h$ at infinity has in the regularity of the solution, providing sharp conditions for which a $BV(\Omega)$ solution exists even for $f$ merely belonging to $L^1(\Omega)$. The simplicity of the methods employed to prove these results makes them powerful and flexible enough to be  concretely  adapted to other problems involving  the $1$-Laplacian operator,  as well as other ones arising from Lagrangian with linear growth (e.g. mean curvature operator, transparent media, and so on).\bk 
\medskip 

The plan of the paper is as follows. In Section \ref{sec_prel} we set the notation and the main functional tools are briefly presented. For the sake of exposition\bk, in Section \ref{sec_pos} we focus on the case of positive data $f$ belonging to $L^N(\Omega)$, while in Section \ref{sec_fnon} we deal with nonnegative $f$'s enjoying the same regularity.  In order to ensure  a clearer  approach to the main ideas of our argument\bk, in both sections we deal with a nonlinearity $h(s)$ with a fixed  behaviour at the origin. In particular, we consider that $h(s)$ near zero is like $s^{-\gamma}$ for some $\gamma>0$. In Section \ref{sec5} we treat the case of a less regular datum $f$, starting in Section~\ref{sec5c} with a datum belonging to the Marcinkiewicz space $L^{N,\infty}(\Omega)$. We continue considering a general $L^m(\Omega)$ ($m\geq 1$) datum and, for the  convenience of the reader, in Section~\ref{sec5a} we  first consider the model problem~\eqref{now}. Then, in Section~\ref{sec5b}, we address a more general nonlinearity $F(x,s)$ as right-hand side without imposing any growth at $s=0$ (see Theorem \ref{teo_ex_F}). Finally, in Section \ref{sec_final} we discuss some further extensions and open problems. Through the construction of explicit examples, we also show the optimality of our assumptions.

\section{Preliminaries and notation}
\label{sec_prel}

In the entire paper $\Omega$ is an open bounded subset of $\R^N$ ($N> 1$) with Lipschitz boundary. For a set $E\subset \R^N$, both $\mathcal H^{N-1}(\partial E)$  and ${\rm Per}(E)$ denote the $(N - 1)$-dimensional Hausdorff measure of the boundary of $E$, while $|E|$ means the classical  $N$-dimensional Lebesgue measure of $E$. 
\medskip

With respect to the functional spaces, we denote by $\mathcal{M}(\Omega)$ the space of Radon measures with finite total variation over $\Omega$. The space of bounded variation functions is denoted by
$$BV(\Omega):= \big\{ u\in L^1(\Omega) : Du \in \mathcal{M}(\Omega)^N \big\},$$ 
where $Du$ stands for the distributional gradient of $u$. Furthermore, we denote $BV_{\rm{loc}}(\Omega)$ as the space of functions such that $u\in BV(\omega)$ for all $\omega\subset \subset \Omega$. We recall that every function $u\in BV(\Omega)$ has a trace defined on $\partial\Omega$ which belongs to $L^1(\partial\Omega)$. We also underline that two equivalent norms in $BV(\Omega)$ are  
$$ \|u\|=\int_\Omega|Du| + \int_\Omega |u|$$
and
$$\displaystyle \|u\|_{BV(\Omega)}=\int_\Omega|Du| + \int_{\partial\Omega}
|u|\, d\mathcal H^{N-1},$$
where  $\int_\Omega |Du|$ denotes the total variation of $Du$ over $\Omega$. \bk
 Let us recall, as  it will be systematically used,  that these norms are lower semicontinuous with respect to the strong convergence in $L^1(\Omega)$. \bk  The precise representative of $u$ is denoted by $u^*$, and it  agrees with the Lebesgue representative of $u$ at the Lebesgue points of $u$, while on the jump set $J_u$ of $u$ it is defined as the algebraic mean of the jump values of $u$ for $\mathcal H^{N-1}$-a.e. $x\in J_u$. 
Let us also recall that, if  $\rho_\epsilon$ is a standard sequence of mollifiers one has the following result.

\begin{proposition}\label{moll}
Let $u \in BV(\Omega)$. Then  $u*\rho_\epsilon$ converges pointwisely to $u^*$ in $\Omega$ up to a set of  $\mathcal H^{N-1}$ null measure. 
\end{proposition}

For more properties concerning  the $BV$ space \bk we refer the reader to \cite{AFP}, from where we mainly set our notations. 
\medskip

Now let us briefly present the $L^\infty$-divergence-measure vector fields theory due to \cite{An} and \cite{CF}. To establish some notations, we denote by $\DM(\Omega)$ the set defined as 
$$\DM(\Omega):= \big\{ z\in L^\infty(\Omega)^N : \operatorname{div}z \in \mathcal{M}(\Omega) \big\},$$ 
and by $\DM_{\rm{loc}}(\Omega)$ the set of the vector fields $z\in L^\infty(\Omega)^N$ such that  $\operatorname{div}z\in \mathcal{M}(\omega)$,  $\forall \omega\subset\!\subset\Omega$.
\medskip

First observe that, if $z\in \DM(\Omega)$, then $\operatorname{div}z$ can be shown to be absolutely continuous with respect to $\mathcal H^{N-1}$ (see~\cite[Proposition~3.1]{CF}). Moreover, a ``dot product'' between a vector field $z\in \DM(\Omega)$ and the gradient $Du$ of a function $u\in BV(\Omega)$ can be defined under some compatibility conditions. This can be done through the distribution $(z,Dv)\colon C^1_c(\Omega)\to \mathbb{R}$ defined as 
\begin{equation}\label{dist1}
	\langle(z,Dv),\varphi\rangle:=-\int_\Omega v^*\varphi\operatorname{div}z-\int_\Omega
	vz\cdot\nabla\varphi, \,\quad \forall \varphi\in C_c^1(\Omega).
\end{equation}

In~\cite{C}, the author shows that~\eqref{dist1} is well defined if $z\in \DM(\Omega)$ and $v\in BV(\Omega) \cap L^\infty(\Omega)$. More general compatibility conditions are proved in~\cite{DGS}, where the authors prove that~\eqref{dist1} is well posed if $z\in \DM_\mathrm{loc}(\Omega)$ and $v\in BV_\mathrm{loc}(\Omega)$ with $v^*\in L^1_\mathrm{loc}(\Omega, \operatorname{div}z)$. In all cases, the distribution defined in~\eqref{dist1} is a Radon measure having locally finite total variation, and  satisfying 
\begin{equation*} 
	|\langle (z, Dv), \varphi\rangle| \le \|\varphi\|_{L^{\infty}(\omega) }\| z
	\|_{L^{\infty}(\omega)^N} \int_{\omega} |Dv|\,,
\end{equation*}
for all open set $\omega \subset\subset \Omega$ and for all $\varphi\in C_c^1(\omega)$. 
\medskip

We observe that, if $u\in BV(\Omega)$, then $(z,Du)$ has finite total variation. In general, for every \(v\in BV_{\rm{loc}}(\Omega)\) the measure $(z, Dv)$ is absolutely continuous with respect to $|Dv|$, so it holds
\begin{equation*} 
	(z, Dv) = \theta(z,Dv,x) \, {|Dv|},
\end{equation*}
where \(\theta(z, Dv, \cdot)\) stands for
the Radon-Nikodym derivative of $(z, Dv)$ with respect to $|Dv|$. Let us state a result which will be used throughout this paper.

\begin{lemma}{\cite[Proposition $4.5$]{CD}}
	\label{lem:Composition}
	Let $z\in \mathcal{D}\mathcal{M}_{\mathrm{loc}}(\Omega)$ and $u\in BV_{\rm loc}(\Omega) \cap L^\infty_\mathrm{loc}(\Omega)$. Let $\Lambda\colon \R\to\R$ be a non-decreasing locally Lipschitz function. Then
	\[
	\theta(z, D\Lambda(u),x) = \theta(z, Du,x) \quad \text{for } |D\Lambda(u)|-{\rm a.e.}\ x \in \Omega.
	\]
	As a consequence, $(z,Du)=|Du|$ as measures implies $(z,D\Lambda(u))= |D\Lambda(u)|$ as measures.
\end{lemma}
\begin{remark}
    \label{rem:Composition}
    From Lemma~\ref{lem:Composition} it can be deduced that if $\Lambda\colon \R\to\R$ is an increasing locally Lipschitz function, then $\theta(z, D\Lambda(u),x) = \theta(z, Du,x)$ for $|Du|$-$\mathrm{a.e.}$ $x \in \Omega$. Then, the reciprocal of Lemma~\ref{lem:Composition} holds, i.e., $(z,D\Lambda(u))= |D\Lambda(u)|$ as measures implies $(z,Du)=|Du|$ as measures. \triang
\end{remark}

Since $\Omega$ has Lipschitz boundary, \bk the outward normal unit vector $\nu(x)$ is defined for $\mathcal H^{N-1}$-almost every $x\in\partial\Omega$. In \cite{An}, it is shown that every $z \in \mathcal{DM}^{\infty}(\Omega)$ possesses
a weak trace on $\partial \Omega$ of the
normal component of  $z$ which is denoted by
$[z, \nu]$. Moreover, it holds
\begin{equation*} 
	\|[z,\nu]\|_{L^\infty(\partial\Omega)}\le \|z\|_{L^\infty(\Omega)^N},
\end{equation*}
and, if $z \in \mathcal{DM}^{\infty}(\Omega)$ and $v\in BV(\Omega)\cap L^\infty(\Omega)$, one has (see \cite[Lemma $5.6$]{C} or \cite[Proposition $2$]{ADS})
\begin{equation}\label{des2}
	v[z,\nu]=[vz,\nu].
\end{equation}

The following Green type formula holds. 

\begin{Proposition}{\cite[Proposition $2.6$]{DGS}}\label{poiu}
	Let $z \in \mathcal{DM}_{\rm{loc}}^{\infty}(\Omega)$. Let $v\in BV(\Omega)\cap L^\infty(\Omega)$ be such that $v^*\in L^1(\Omega, {\rm div} z)$.
	Then $vz\in \mathcal{DM}^{\infty}(\Omega)$ and the following  holds
	\begin{equation}\label{green}
		\int_{\Omega} v^* \,{\rm div} z  + \int_{\Omega} (z, Dv) =
		\int_{\partial \Omega} [vz, \nu] \ d\mathcal H^{N-1}\,.
	\end{equation}
\end{Proposition}

 Finally, \bk if $z \in \DM_{\rm{loc}}(\Omega)$ and $v\in BV(\Omega)\cap L^\infty(\Omega)$ are such that $vz\in \mathcal{DM}^{\infty}(\Omega)$, it can be proved (see \cite[Proposition 2.7]{DGS}) that 
\begin{equation}
	\label{p27} |[vz,\nu]|\le \left|v_{\res {\partial\Omega}} \right|\,\|z\|_{L^\infty(\Omega)^N}\,\quad\mathcal H^{N-1}\hbox{-a.e. on }\partial\Omega\,.
\end{equation}
\bk

{\bf Notation.}  We denote by $\mathcal{S}_1$ the best constant in the Sobolev inequality for functions in $BV(\Omega)$, that  is
\[
\|v\|_{L^{\frac{N}{N-1}}(\Omega)} \le \mathcal{S}_1 \|v\|_{BV(\Omega)}, \ \ \forall v \in BV(\Omega).
\]
Recall that $\mathcal{S}_1=\big(N\omega_{N}^{\frac{1}{N}}\big)^{-1}\,,$
where $\omega_{N}$ is the volume of the unit sphere of $\rn$ (see for instance \cite{talenti}).

The following functions will be widely used in the sequel: for a fixed $k>0$, $T_{k}$ and $G_{k}$ are the functions defined by
$$
T_k(s)=\max (-k,\min (s,k)) \qquad \text{and} \qquad G_k(s)=(|s|-k)^+ \operatorname{sign}(s). 
$$
Note in particular that   $T_k(s) + G_k(s)=s$, for any $s\in \mathbb{R}$. Another useful function will be $S_\delta$, which is defined as 
\begin{align}\label{Sdelta}
	\displaystyle
	S_\delta(s):=
	\begin{cases}
		0 \ \ &s\le \delta, \\
		\displaystyle\frac{s-\delta}{\delta} \ \ &\delta <s< 2\delta, \\
		1 \ \ &s\ge 2\delta.
	\end{cases}
\end{align}

Regarding the integrals notation, if there is no ambiguity we will use the notation
$$
\int_\Omega f:=\int_\Omega f(x)\,dx
$$
and, if $\mu$ is a Radon measure,
\[\int_\Omega f\mu:=\int_\Omega f\, d\mu\,.\]

\medskip

To conclude, we denote by $C$ several constants whose value may change from line to line and, sometimes, on the same line. These values only depend on the data but they do not depend on the indexes of the sequences. We underline the use of the standard convention to do not relabel an extracted compact subsequence.

\section{The case of  a positive datum $f\in L^{N}(\Omega)$}
\label{sec_pos}

In this section  we deal with the existence of nonnegative solutions to 
\begin{equation}
	\label{eq:PbMain}
	\begin{cases}
		\dis -\Delta_1 u = h(u)f & \text{in}\;\Omega,\\
		u=0 & \text{on}\;\partial\Omega,
	\end{cases}
\end{equation}
where  $\Delta_1 u= \operatorname{div}(|D u|^{-1}D u)$ \bk is the $1$-Laplace operator, $f\in L^N(\Omega)$ is positive in $\Omega$,  and $h\colon [0,\infty)\mapsto [0,\infty]$ is assumed to be a continuous function finite outside the origin satisfying
\begin{equation}\label{eq:hyp_h}
	\displaystyle \exists\; c_1, s_1, \gamma> 0\;\ \text{such that}\;\  h(s)\le \frac{c_1}{s^\gamma} \ \text{ if } s\leq s_1, \ \ \ h(\infty):=\limsup_{s\to\infty}h(s) < \infty.
\end{equation}
Let us stress that for the entire paper we use the notation \begin{equation*}
	\displaystyle h_k(\infty):=\sup_{s\in [k,\infty)}h(s).
\end{equation*}

Let us first precise what we mean by a solution to \eqref{eq:PbMain} in this case.

\begin{defin}\label{def_fpos}
		Let $0< f\in L^N(\Omega)$. A nonnegative function $u\in BV_{\rm loc}(\Omega)$ is a solution to problem~\eqref{eq:PbMain} if there exists $z\in \mathcal{D}\mathcal{M}^\infty(\Omega)$ with $\|z\|_{L^\infty(\Omega)^N}\le 1$ such that 
	\begin{gather}        
        \label{def:distrp=1}
		-\operatorname{div}z = h(u)f \ \ \text{as measures in }\Omega, \\[1.5mm]
        \label{def:zp=1}
 		(z,DT_k(u))=|DT_k(u)| \ \ \text{as measures in } \Omega \text{ (for every } k>0),
	\end{gather}        
	and one of the following conditions holds:
	\begin{equation}  
        \label{def:bordo}
        \lim_{\varepsilon\to 0^+}\fint_{\Omega\cap B(x,\epsilon)} u (y) dy = 0 \ \ \ \text{or} \ \ \ [z,\nu] (x)= -1 \ \ \ \text{for  $\mathcal{H}^{N-1}$-a.e. } x \in \partial\Omega.			
	\end{equation}
\end{defin}

\begin{remark}\label{remarkDefinizione}
Let us underline that the above definition is essentially the one given in \cite[Definition $6.1$]{DGOP}. The only difference relies on the fact that here we do not require $u$ to be bounded as we aim to deduce it as a regularity property (see Theorem \ref{teo_reg} below). Let us also recall that \eqref{def:distrp=1} and \eqref{def:zp=1} is the way in which the equation in \eqref{eq:PbMain} is meant, while  \eqref{def:bordo} is the weak way in which the boundary datum is intended, which is nowadays classical for these kind of problems. Indeed, if $u\in BV(\Omega)$ (i.e. $u$ admits a trace on $\partial\Omega$) this condition turns out to be clearly equivalent to the well-known boundary condition $u(x)([z,\nu] (x)+1)=0$ for  $\mathcal{H}^{N-1}$-a.e. $x \in \partial\Omega$ (see \cite{ACM}). 
\triang
\end{remark}

\begin{remark}
\label{este}
Let us also remark that one can extend the set of functions to test~\eqref{def:distrp=1}. In fact, once~\eqref{def:distrp=1} holds, then~\cite[Lemma 5.3]{DGOP} implies that $ \operatorname{div}z \in L^1(\Omega)$ and  
\begin{equation*}
-\int_\Omega v\operatorname{div}z  = \int_\Omega h(u) f v\,  \ \ \text{for all}\ v\in BV(\Omega)\cap L^{\infty}(\Omega).
\end{equation*} 
\triang
\end{remark}

Let us state the main result of this section. 

\begin{theorem}\label{teo_reg}
	Let $h$ satisfy \eqref{eq:hyp_h} and let $0<f\in L^N(\Omega)$ be such that \begin{equation}\label{cond1}\|f\|_{L^N(\Omega)} < \left(\mathcal{S}_1h(\infty)\right)^{-1}.\end{equation} 	Let $u$ be a solution in the sense of Definition \ref{def_fpos} such that $u^\alpha\in BV(\Omega)$ for some $\alpha>0$. Then $u\in BV(\Omega) \cap L^\infty(\Omega)$.
\end{theorem}

As a consequence of the above result and of the existence results given in \cite{DGOP}, we give the following existence theorem.
\begin{theorem}\label{teo_ex}
		Let $h$ satisfy \eqref{eq:hyp_h} and let $0<f\in L^N(\Omega)$ be such that \eqref{cond1} holds. 
Then there exists a solution $u\in BV(\Omega) \cap L^\infty(\Omega)$ to \eqref{eq:PbMain}  in the sense of Definition \ref{def_fpos}\bk. Moreover, if $h$ is decreasing then $u$ is the unique solution of \eqref{eq:PbMain}  in $BV(\Omega)$. 
\end{theorem}

We first prove the regularity Theorem \ref{teo_reg}.

\begin{proof}[Proof of Theorem \ref{teo_reg}] We split the proof into two steps. 

{\bf Step 1.} We first show that $u$ belongs to $L^\infty(\Omega)$. Let $\ell>k>0$ and let us observe that $G_k(T_\ell(u))\in BV(\Omega)$. Indeed, as 
\begin{equation*}
	H_{k,\ell}(s):=
	\begin{cases}
		0 & s<k^\alpha,\\
		s^{\frac{1}{\alpha}}-k & k^\alpha\leq s\leq \ell^\alpha,\\[0.5mm]
		\ell-k & s> \ell^\alpha,
	\end{cases}
\end{equation*}
is a Lipschitz function, then $u^\alpha\in BV(\Omega)$ implies $G_k(T_\ell(u))=H_{k,\ell}(u^\alpha)\in BV(\Omega)$.

Recalling Remark~\ref{este}, we take $G_k(T_\ell(u))\in BV(\Omega)\cap L^\infty(\Omega)$ as test function in \eqref{eq:PbMain}, we use the   H\"{o}lder and the Sobolev inequalities, \bk  to get
\begin{equation}\label{fpos1}
	\begin{aligned}
-\int_\Omega G_k(T_\ell(u)) \operatorname{div} z &= \int_\Omega h(u)f G_k(T_\ell(u)) \leq h_k(\infty) \|f\|_{L^{N}(\Omega)} \|G_k(T_\ell(u))\|_{L^{\frac{N}{N-1}}(\Omega)} 
\\
&\leq h_k(\infty) \|f\|_{L^{N}(\Omega)} \mathcal{S}_1 \left(\int_\Omega |DG_k(T_\ell(u))| + \int_{\partial\Omega} G_k(T_\ell(u))\, d \mathcal{H}^{N-1}\right).
	\end{aligned}
\end{equation}

Now observe that it follows from \eqref{def:bordo} (also recall Remark \ref{remarkDefinizione}) that $G_k(T_\ell(u))(1+[z,\nu])=0$ $\mathcal{H}^{N-1}$-$\mathrm{a.e.}$ on $\partial\Omega$, that is, $-G_k(T_\ell(u)) [z,\nu] =  G_k(T_\ell(u))$ $\mathcal{H}^{N-1}$-a.e. on $\partial\Omega$. Therefore, using \eqref{green}, \eqref{des2} and Lemma \ref{lem:Composition}   we obtain that the left-hand of \eqref{fpos1} can be written as
	\begin{equation}\label{fpos2}
		\begin{aligned}
	-\int_\Omega G_k(T_\ell(u)) \operatorname{div}z &= \int_\Omega (z, DG_k(T_\ell(u)) ) - \int_{\partial\Omega} G_k(T_\ell(u)) [z,\nu] \, d\mathcal H^{N-1}
	\\
	&= \int_\Omega |DG_k(T_\ell(u))| + \int_{\partial\Omega} G_k(T_\ell(u)) \, d \mathcal{H}^{N-1}.
	\end{aligned}
	\end{equation}
	Hence, gathering \eqref{fpos2} into \eqref{fpos1}, one yields to
	\begin{equation}\label{stimaGK}
	(1-h_k(\infty) \|f\|_{L^N(\Omega)} \mathcal{S}_1) \left(\int_\Omega |DG_k(T_\ell(u))| + \int_{\partial\Omega} G_k(T_\ell(u))\, d \mathcal{H}^{N-1}\right) \leq 0.
	\end{equation}
	
	Then, as \eqref{cond1} holds, we can fix  $k= \tilde{k}$ large enough such that $1-h_{\tilde{k}}(\infty) \|f\|_{L^N(\Omega)} \mathcal{S}_1>0$, and we deduce that
	\begin{equation*} 
	\int_\Omega |DG_{\tilde{k}}(T_\ell(u))| + \int_{\partial\Omega} G_{\tilde{k}}(T_\ell(u))\,d \mathcal{H}^{N-1}  = 0.
	\end{equation*}

This gives us that $\|G_{\tilde{k}}(T_\ell(u)) \|_{BV(\Omega)}=0$ and then $G_{\tilde{k}}(T_\ell(u))=0$ $\mathrm{a.e.}$ in $\Omega$ for some $\ell>\tilde{k}$ which can be seen as fixed. Therefore $\|u\|_{L^\infty(\Omega)} \leq \tilde{k}$.

\medskip

{\bf Step 2.} Let us prove that $u\in BV(\Omega)$. 
First observe that $v_n=\left(u^\alpha + \frac{1}{n} \right)^\frac{1}{\alpha} - \|u\|_{L^\infty(\Omega)}-1 \in BV(\Omega)\cap L^\infty(\Omega)$ is nonpositive for $n$ large enough and it satisfies 
$$
\|v_n\|_{L^\infty(\Omega)}\leq \|u\|_{L^\infty(\Omega)} +1\,.  
$$
Then taking $v_n$ as test function in \eqref{eq:PbMain} one can drop the right-hand side and,   after applying \eqref{green} and \eqref{des2}, one has
\[
\int_\Omega (z,Dv_n) \leq \int_{\partial\Omega} v_n[z,\nu]\, d\mathcal H^{N-1} \leq (\|u\|_{L^\infty(\Omega)} + 1) \mathcal{H}^{N-1}(\partial\Omega).
\]

As $(z,Dv_n) = |Dv_n|$ by Lemma \ref{lem:Composition}, we get
\[
\int_\Omega |Dv_n| \leq (\|u\|_{L^\infty(\Omega)} + 1)\mathcal{H}^{N-1}(\partial\Omega).
\]

Due to $v_n\to u-\|u\|_{L^\infty(\Omega)}-1$ in $L^1(\Omega)$ by the Lebesgue Theorem, using the lower semicontinuity we deduce
\begin{equation}\label{stimaBV}
\int_\Omega |Du| = \int_\Omega |D(u-\|u\|_{L^\infty(\Omega)}-1)| \leq (\|u\|_{L^\infty(\Omega)} + 1)\mathcal{H}^{N-1}(\partial\Omega)
\end{equation}
which implies that $u\in BV(\Omega)$. This concludes the proof. \bk
\end{proof}

Now we are in position to prove the existence Theorem \ref{teo_ex}.

\begin{proof}[Proof of Theorem \ref{teo_ex}]
	 First observe that \cite[Theorems $3.3$ and $3.4$]{DGOP} give that there exists a bounded solution $u$ to \eqref{eq:PbMain} in the sense of Definition \ref{def_fpos} and such that $u^{\max(1,\gamma)}\in BV(\Omega)$. Moreover, if $\gamma>1$, Theorem \ref{teo_reg} implies that $u\in BV(\Omega)$.  Uniqueness follows by applying \cite[Theorem 3.5]{DGOP}. This concludes the proof.  
\end{proof}

\section{A nonnegative datum $f$}
\label{sec_fnon}

Here we deal with the case of a nonnegative $f\in L^N(\Omega)$; in particular we show that there exists a global $BV$ solution to \eqref{eq:PbMain} under the natural smallness condition on $f$ for any $\gamma>0$.
\medskip

The main difference with respect to the case of a positive $f$ is that the solutions are not necessarily expected to be positive.
Following \cite{DGOP} we  precise what is a solution to \eqref{eq:PbMain} in the significant case in which  $f$ is  nonnegative and $h(0)=\infty$. Otherwise, if $h(0)<\infty$, the results of Section \ref{sec_pos}  apply \bk using Definition \ref{def_fpos}.
\medskip

 As in the previous section, we are controlling the growth of $h$ near zero by hypothesis~\eqref{eq:hyp_h}. For the rest of the paper, we set following notation:\bk
\begin{equation}\label{sigma}
	\sigma:=\max(1,\gamma).
\end{equation}

\begin{defin}\label{def_fnon} 
	Let $0\leq f\in L^N(\Omega)$. A nonnegative function $u\in BV_{\rm loc}(\Omega)$ with $ T_k(u)^\sigma\in BV(\Omega)$ and $\chi_{\{u>0\}} \in BV_{\rm loc}(\Omega)$ is a solution to problem~\eqref{eq:PbMain} if $h(u)f\in L^{1}_{\rm loc} (\Omega)$ and if there exists $z\in \mathcal{D}\mathcal{M}^\infty_{\rm loc}(\Omega)$ with $\|z\|_{L^\infty(\Omega)^N}\le 1$ such that
	\begin{gather}
		\label{def:distrp=1non}
		-(\operatorname{div}z) \chi^*_{\{u>0\}} = h(u)f \ \ \text{as measures in }\Omega, \\[1mm]
		\label{def:zp=1non}
		(z,DT_k(u))=|DT_k(u)| \ \ \text{as measures in } \Omega \text{ (for every } k>0),\\[1mm]
		\label{def:bordonon}
		 T_k(u)^\sigma(x)+[T_k(u)^\sigma z,\nu](x)= 0 \ \ \ \text{for  $\mathcal{H}^{N-1}$-a.e. } x \in \partial\Omega.
	\end{gather}
\end{defin}
\begin{remark}
Once again, as we want to derive the  boundedness of the solution as a regularity property, the request $T_k(u)^\sigma \in BV(\Omega)$ is natural in order to give sense to the boundary condition since, as $z$ only belongs to $\DM_{\rm{loc}}(\Omega)$,   then  $u^{\sigma}z$ may have no trace a priori and equality $[T_k(u)^\sigma z,\nu] = T_k(u)^\sigma [z,\nu]$ is not  satisfied in general.  It is worth mentioning that $T_k(u)^\sigma \in BV(\Omega)$ and equality~\eqref{def:distrp=1non} are enough to show that $(T_k(u)^\sigma)^* \in L^1(\Omega, \operatorname{div}z)$ and then the existence of a normal trace on $\partial\Omega$ for $T_k(u)^\sigma z$ is just a consequence of Proposition~\ref{poiu}.  
Finally observe that  $h(u)f\in L^{1}_{\rm loc} (\Omega)$ is needed since it is not immediately implied from \eqref{def:distrp=1non}. 
\triang
\end{remark}

\begin{remark}
	We also  observe that a solution in the sense of Definition \ref{def_fnon} is also a solution in the sense of Definition \ref{def_fpos} provided  $f$ is positive. Indeed, as $h(0)=\infty$ and $h(u)f \in L^1_{\rm loc}(\Omega)$, one has that $u>0$ in $\Omega$, and thus  from \eqref{def:distrp=1non} one gets \eqref{def:distrp=1}. Moreover, as one can apply~\cite[Lemma $2.3$]{gop}, $z\in \DM(\Omega)$. Then, thanks also to \eqref{des2}, the boundary condition \eqref{def:bordo} is satisfied. 
 \triang
\end{remark}

In the next result we show the existence of a global bounded variation solution to \eqref{eq:PbMain}.

\begin{theorem}\label{teo_exnon}
	Let $h$ satisfy \eqref{eq:hyp_h}  such that $h(0)=\infty$, and   let $0\leq f\in L^N(\Omega)$ be such that \begin{equation}\label{cond2}\|f\|_{L^N(\Omega)} < \left(\mathcal{S}_1h(\infty)\right)^{-1}.\end{equation} Then there exists a solution $u \in BV(\Omega)\cap L^\infty(\Omega)$ to \eqref{eq:PbMain}.
\end{theorem}

\begin{proof}
	First observe that, if $\gamma\le 1$, the existence of a bounded solution belonging to $BV(\Omega)$ follows from \cite[Theorem 6.4]{DGOP}. Hence, from here on, we suppose $\gamma>1$.
	\medskip
	
	Let us work by approximation through the following problems 
	\begin{equation}
		\label{eq:PbApprox}
		\begin{cases}
			\dis -\Delta_1 u_n = h(u_n)f_n & \text{in}\;\Omega,\\
			u_n=0 & \text{on}\;\partial\Omega,
		\end{cases}
	\end{equation}
	where $f_n:= \max (f, \frac{1}{n})$. Observe that $f_n\geq \frac{1}{n}$ almost everywhere in $\Omega$ and that $f_n\to f$ in $L^N(\Omega)$ so that there exists some $n_0\in \mathbb{N}$ such that $\|f_n\|_{L^N(\Omega)} < (\mathcal{S}_1h(\infty))^{-1}$ for every $n\geq n_0$ (recall \eqref{cond2}).
	Therefore, for $n\geq n_0$, the existence of a positive solution $u_n\in BV(\Omega)\cap L^\infty(\Omega)$ to \eqref{eq:PbApprox} follows from Theorem \ref{teo_ex}. We denote by $z_n$ the vector field associated to $u_n$. From now on, we will always consider $n\geq n_0$ even if it is not specified.
	\medskip 	
	\textbf{Step 1:} $u_n$ is bounded in $L^\infty(\Omega)$ with respect to $n$.
	
	Taking $G_k(u_n)\in BV(\Omega)\cap L^\infty(\Omega)$ ($k>0$) as test function in \eqref{eq:PbApprox} and arguing as for obtaining \eqref{stimaGK}, one has 
	\[
	(1-h_k(\infty) \|f_n\|_{L^N(\Omega)} \mathcal{S}_1)\left(\int_\Omega |DG_k(u_n)| + \int_{\partial\Omega} G_k(u_n)\, d \mathcal{H}^{N-1}\right) \leq 0.
	\]
	Since $\|f_n\|_{L^N(\Omega)}$ is a  non-increasing \bk sequence and $\|f_{n_0}\|_{L^N(\Omega)} < (\mathcal{S}_1h(\infty))^{-1}$, one can fix $k=\tilde{k}$ large enough so that $\|f_{n}\|_{L^N(\Omega)} < (\mathcal{S}_1 h_{\tilde{k}}(\infty))^{-1}$ for all $n\geq n_0$. Therefore, one gains 
	\[
	\int_\Omega |DG_{\tilde{k}}(u_n)| + \int_{\partial\Omega} G_{\tilde{k}}(u_n)\, d \mathcal{H}^{N-1} = 0,
	\]
	namely $\|G_{\tilde{k}}(u_n)\|_{BV(\Omega)} = 0$. This implies that $G_{\tilde{k}}(u_n)=0$ in $\Omega$ and thus $\|u_n\|_{L^\infty(\Omega)}\leq {\tilde{k}}$ for all $n\ge n_0$.
	
	\medskip 	

	\textbf{Step 2:} $u_n$ is bounded in $BV(\Omega)$ with respect to $n$. 
	
	Here one can reason as for deducing \eqref{stimaBV}; hence, after taking $v_{n,\varepsilon}=\left(u_n^\gamma + \varepsilon \right)^\frac{1}{^\gamma} - \|u_n\|_{L^\infty(\Omega)}-1 \in BV(\Omega)\cap L^\infty(\Omega)$ ($\varepsilon$ small enough) as test function into \eqref{eq:PbApprox} and reasoning analogously to the Step 2 of Theorem \ref{teo_ex}, one yields to
	\[
	\int_\Omega |Dv_{n,\varepsilon}| \leq (\|u_n\|_{L^\infty(\Omega)} + 1) \mathcal{H}^{N-1}(\partial\Omega).
	\]
	Since $v_{n,\varepsilon} \to u_n-\|u_n\|_{L^\infty(\Omega)}-1$ in $L^1(\Omega)$ by the Lebesgue Theorem as $\varepsilon \to 0$, using the lower semicontinuity one has
	\[
	\int_\Omega |Du_n| = \int_\Omega |D(u_n-\|u_n\|_{L^\infty(\Omega)}-1)| \leq (\|u_n\|_{L^\infty(\Omega)} + 1)\mathcal{H}^{N-1}(\partial\Omega) \leq C,
	\]
	since it follows from Step $1$ that $u_n$ is bounded in $L^\infty(\Omega)$ with respect to $n$. Thus, this shows that $u_n$ is bounded in $BV(\Omega)$ with respect to $n$.
	
	As $u_n$ is bounded in $BV(\Omega)$, there exists a subsequence, still denoted by $u_n$,  such that $u_n\rightarrow u$  in $L^1(\Omega)$. and $\mathrm{a.e.}$ in $\Omega$ for some $u\in BV(\Omega)$. Since $u_n$ is uniformly bounded in $L^\infty(\Omega)$, we also have $u\in L^\infty(\Omega)$ and $u_n\to u$ in $L^q(\Omega)$ for every $q<\infty$.\bk 
	
	Furthermore, since $\|z_n\|_{L^\infty(\Omega)^N}\leq 1$, then there exists $z\in L^\infty(\Omega)^N$ such that $z_n\rightharpoonup z$ *-weakly in $L^\infty(\Omega)^N$. Since the norm is *-weakly lower semicontinuous, then we have also $\|z\|_{L^\infty(\Omega)^N}\leq 1$.
	
	\medskip 	
	\textbf{Step 3:} $h(u)f\in L^1_\mathrm{loc}(\Omega)$.

	Taking $0\leq \varphi \in C_c^1(\Omega)$ as test function in \eqref{eq:PbApprox}, one gets that
	\begin{equation*}
		\int_\Omega z_n\cdot \nabla \varphi= \int_\Omega h(u_n)f_n\varphi.
	\end{equation*}
	Then one can pass to the limit with respect to $n$ the previous; in particular,  using that $z_n\rightharpoonup z$ *-weakly in $L^\infty(\Omega)^N$  on the left-hand,  and the Fatou Lemma on the right-hand, one obtains
	\begin{equation}
		\label{eq:Proof_f>=0_1}
		\int_\Omega h(u)f\varphi \leq \int_\Omega z\cdot \nabla \varphi = - \int_\Omega \varphi \operatorname{div}z,\ \forall \varphi \in C_c^1(\Omega) \text{ with }\varphi \ge 0.
	\end{equation}
	This implies that $h(u)f\in L^1_\mathrm{loc}(\Omega)$ and then $\{u=0\}\subseteq \{f=0\}$ up to a set of zero Lebesgue measure.  From this inequality it can also be deduced that $z\in \DM_{\mathrm{loc}}(\Omega)$.  
	
	\medskip 	
	\textbf{Step 4:} Proof of \eqref{def:distrp=1non}.
	
	We take $S_\delta(u_n)\varphi \in BV(\Omega)\cap L^\infty(\Omega)$ ($S_\delta$ is defined in \eqref{Sdelta}) with $0\leq \varphi \in C_c^1(\Omega)$ as test function in \eqref{eq:PbApprox} obtaining (recall \eqref{dist1})
	\begin{equation}
		\label{eq:Proof_f>=0_2}
		\int_\Omega (z_n, DS_\delta(u_n)) \varphi + \int_\Omega z_n \cdot \nabla \varphi \ S_\delta(u_n)= \int_\Omega h(u_n)f_n S_\delta(u_n) \varphi.
	\end{equation}
	Since $S_\delta(s)$ is a non-decreasing function with respect to  $s$ and since $(z_n,Du_n) = |Du_n|$ as measures in $\Omega$, it follows from Lemma \ref{lem:Composition} that $(z_n,DS_\delta(u_n)) = |DS_\delta(u_n)|$ as measures in $\Omega$. 
	
	Now, we want to pass first to the limit in $n$. Since $S_\delta(u_n)\to S_\delta(u)$ in $L^1(\Omega)$, in the first term we use the lower semicontinuity. In the second term we can easily pass to the limit since $z_n\to z$ *-weakly in  $L^\infty(\Omega)^N$ \bk and $S_\delta(u_n)\to S_\delta(u)$ in $L^1(\Omega)$. Finally, in the last term we can use the Lebesgue Theorem since
	\[
	h(u_n)f_n S_\delta(u_n) \varphi \leq \sup_{s\in[\delta,\infty)} h(s) \ (f+1) \varphi \in L^1(\Omega).
	\]
	Therefore, passing to the limit with respect to $n$ into \eqref{eq:Proof_f>=0_2}, one deduces
	\begin{equation*}
		\int_\Omega |DS_\delta(u)| \varphi + \int_\Omega z \cdot \nabla \varphi \ S_\delta(u) \leq \int_\Omega h(u)f S_\delta(u) \varphi.
	\end{equation*}
	
	Now we pass to the limit the previous in $\delta$. Due to $\|z\|_{L^\infty(\Omega)^N}\leq 1$ and~\eqref{eq:Proof_f>=0_1}, we can use the Lebesgue Theorem to pass to the limit in the last two terms. Since $S_\delta(u)\to \chi_{\{u>0\}}$ in $L^1(\Omega)$, we can use the lower semicontinuity on the first term to obtain
	\begin{equation}
		\label{eq:Proof_f>=0_4}
		\int_\Omega |D\chi_{\{u>0\}}| \varphi + \int_\Omega z \cdot \nabla \varphi \ \chi_{\{u>0\}} \leq \int_\Omega h(u)f \chi_{\{u>0\}} \varphi.
	\end{equation}
	In particular, we also have 
	\[
	\chi_{\{u>0\}}\in BV_\mathrm{loc}(\Omega).
	\]
	
	Now observe that  $(z, D\chi_{\{u>0\}}) \le  |D\chi_{\{u>0\}}|$ since $\|z\|_{L^\infty(\Omega)^N}\le 1$, then  from~\eqref{eq:Proof_f>=0_4} and from \eqref{dist1}, one deduces 
	\begin{equation}
		\label{eq:Proof_f>=0_5}
		- \int_\Omega \varphi \chi^*_{\{u>0\}} \operatorname{div} z \leq \int_\Omega h(u)f \chi_{\{u>0\}} \varphi \leq \int_\Omega h(u)f \varphi.
	\end{equation} \bk
	
	To prove the reverse inequality, we take $\varphi = (\chi_{\{u>0\}} * \rho_\varepsilon)\phi$ with $0\leq \phi \in C^1_c(\Omega)$ in~\eqref{eq:Proof_f>=0_1}, where $\rho_\varepsilon$ is a mollifier. Using the Lebesgue Theorem in both sides (recall that $h(u)f\in L_\mathrm{loc}^1(\Omega)$), we can pass to the limit in $\varepsilon$ to get, taking into account Proposition~\ref{moll}, that
	\begin{equation*}
		\int_\Omega h(u)f \chi_{\{u>0\}} \phi \leq - \int_\Omega \phi \chi^*_{\{u>0\}}\operatorname{div}z.
	\end{equation*}
	Since $\{u=0\}\subseteq \{f=0\}$, we have 
	\begin{equation}
		\label{eq:Proof_f>=0_6}
		\int_\Omega h(u)f \phi \leq - \int_\Omega \phi \chi^*_{\{u>0\}}\operatorname{div}z, \text{ with }\varphi\ge 0.
	\end{equation}
	Joining~\eqref{eq:Proof_f>=0_5} and~\eqref{eq:Proof_f>=0_6} we conclude that 
	\begin{equation*}
		-\int_\Omega \varphi \chi^*_{\{u>0\}} \operatorname{div} z = \int_\Omega h(u)f \varphi,\ \forall \varphi\in C_c^1(\Omega), \text{ with }\varphi\ge 0,
	\end{equation*} \bk
	which gives \eqref{def:distrp=1non}.
	
	\medskip
	
	\textbf{Step 5:} Proof of \eqref{def:zp=1non}.
	
	Let us take $u_n^\gamma \varphi \in BV(\Omega)\cap L^\infty(\Omega)$ with $0\leq \varphi\in C_c^1(\Omega)$ as test function in \eqref{eq:PbApprox} to deduce
	\[
	-\int_\Omega u_n^\gamma \varphi\operatorname{div} z_n = \int_\Omega h(u_n) u_n^\gamma f_n \varphi.
	\]
Then it follows from \eqref{dist1} that it holds
	\begin{equation}
		\label{eq:Proof_f>=0_7}
		\int_\Omega (z_n,Du_n^\gamma) \varphi + \int_\Omega u_n^\gamma\, z_n\cdot \nabla\varphi = \int_\Omega h(u_n) u_n^\gamma f_n \varphi,
	\end{equation}
	and we want to pass to the limit the previous in $n$.
	
	Now observe that Lemma \ref{lem:Composition} gives that $(z_n,Du_n^\gamma)=|Du_n^\gamma|$ holds as measures in $\Omega$. Moreover, the second term converges   since $z_n\rightharpoonup z$ $*$-weakly in $L^\infty(\Omega)^N$ and $u_n^\gamma\to u^\gamma$ in $L^1(\Omega)$. Observe also that the term on the right-hand of \eqref{eq:Proof_f>=0_7} converges since $u_n$ is bounded in $L^\infty(\Omega)$ with respect to $n$ and $h(s)s^\gamma \le c_1$ for $s\le s_1$ so that one can apply the Lebesgue Theorem. Lastly, thanks also to the lower semicontinuity on the first term into \eqref{eq:Proof_f>=0_7}, one obtains
	\begin{equation}
		\label{eq:Proof_f>=0_8}
		\int_\Omega \varphi |Du^\gamma| \leq - \int_\Omega u^\gamma\, z\cdot \nabla\varphi + \int_\Omega h(u)u^\gamma f\varphi.
	\end{equation}
	
	Now note that, by a mollification argument in~\eqref{eq:Proof_f>=0_1}, one can also show that 
	\begin{equation}
 	\label{eq:Proof_f>=0_9}
	h(u) u^\gamma f \leq  -(u^\gamma)^* \operatorname{div}z \ \ \text{as measures in }\Omega,
	\end{equation}
	which, used into \eqref{eq:Proof_f>=0_8} and recalling \eqref{dist1}, gives that 
	\[
	\int_\Omega \varphi |Du^\gamma| \leq \int_\Omega (z,Du^\gamma)\varphi,\ \forall \varphi\in C_c^1(\Omega), \text{ with }\varphi \ge 0.
	\]
	Since the reverse inequality is immediate from $\|z\|_{L^\infty(\Omega)^N}\leq 1$, one yields to
	\begin{equation*}
		\int_\Omega \varphi |Du^\gamma| = \int_\Omega (z,Du^\gamma)\varphi,\ \forall \varphi\in C_c^1(\Omega).
	\end{equation*}

	Therefore, $(z, Du^\gamma) = |Du^\gamma|$ as measures in $\Omega$. Since $s\mapsto s^\gamma$ is increasing, an application of Remark~\ref{rem:Composition} gives both that $(z,Du)=|Du|$ as measures in $\Omega$ and  that~\eqref{def:zp=1non} holds.

	\medskip 
	\textbf{Step 6:} Proof of \eqref{def:bordonon}.
	
	Let us take $T_k(u_n)^\gamma \in BV(\Omega) \cap L^\infty(\Omega)$ ($k>0$) as test function in \eqref{eq:PbApprox} in order to get
	\[
	-\int_\Omega T_k(u_n)^\gamma \operatorname{div} z_n = \int_\Omega h(u_n) T_k(u_n)^\gamma f_n.
	\]
	
	 Reasoning as in  Remark \ref{remarkDefinizione} one readily  has  that $-T_k(u_n)^\gamma [z_n,\nu] = T_k(u_n)^\gamma$\  $\mathcal{H}^{N-1}$-a.e.   on $\partial\Omega$. Moreover, Lemma~\ref{lem:Composition} can be applied to deduce that $(z_n, DT_k(u_n)^\gamma) =  |DT_k(u_n)^\gamma| \bk$ as measures in $\Omega$. Thus, an application of~\eqref{green} gives  
	\[
	\int_\Omega |DT_k(u_n)^\gamma| + \int_{\partial \Omega} T_k(u_n)^\gamma\, d\mathcal{H}^{N-1} = \int_\Omega h(u_n) T_k(u_n)^\gamma f_n.
	\]
	Once more we can pass to the limit as $n\to\infty$ the right-hand of the previous thanks to an application of the Lebesgue Theorem. Observe that for the left-hand one can use the lower semicontinuity of the norm with respect to $n$, yielding to  
	\[\begin{aligned}
	\int_\Omega |D T_k(u)^\gamma| + \int_{\partial\Omega} T_k(u)^\gamma \, d\mathcal{H}^{N-1} &\leq \int_\Omega h(u)T_k(u)^\gamma f \leq \bk -\int_\Omega (T_k(u)^\gamma)^* \operatorname{div}z \\ &\stackrel{\eqref{green}}{=} \int_\Omega (z,DT_k(u)^\gamma) - \int_{\partial\Omega} [ T_k(u)^\gamma z,\nu]\, d\mathcal{H}^{N-1},
	\end{aligned}\]
	where the second inequality can be deduced as for~\eqref{eq:Proof_f>=0_9}.
	
	Now, since $(z,DT_k(u)^\gamma) = |DT_k(u)^\gamma|$ as measures in $\Omega$, one has
	\[
	\int_{\partial\Omega} (T_k(u)^\gamma +[T_k(u)^\gamma z,\nu])\, d\mathcal{H}^{N-1} \leq 0.
	\]
	Finally, note that $|[T_k(u)^\gamma z,\nu]|\leq {T_k(u)^\gamma}_{\res {\partial\Omega}}$ due to~\eqref{p27}, so one gets that $T_k(u)^\gamma (x)+[T_k(u)^\gamma z,\nu](x)=0$ for $\mathcal{H}^{N-1}$-a.e. $x \in \partial\Omega$.  This concludes the proof. 
\end{proof}

\section{Less regular data and  a generic nonlinearity}\label{sec5}

In this section we investigate problems whose prototype is  \eqref{eq:PbMain}  in presence of a datum $f$ which is  increasingly less regular. For the sake of generality, in Section \ref{sec5b}, we also consider the case of a  general nonlinearity $F(x,u)$.

\subsection{The case $f\in L^{N,\infty}(\Omega)$}	
\label{sec5c}

The main theorems proved in the previous sections can be  extended in a quite plain way to the case of a nonnegative data $f$ belonging to the Lorentz space  $L^{N,\infty}(\Omega)$, also called in literature Marcinkiewicz space $M^{N}(\Omega)$ or $L^{N}$-weak Lebesgue space. For an introduction to these spaces, we refer the interested reader to \cite{PKF}. 
\medskip

To our aim, it is sufficient to recall that a  H\"older inequality is available and that the conjugate space associated to $L^{p,q}(\Omega)$ for $p>1$ and $q\in [1,\infty]$ is $L^{p',q'}(\Omega)$. Moreover, a Sobolev inequality holds, that is
\begin{equation}\label{soblor}
	\|v\|_{L^{\frac{N}{N-1},1}(\Omega)}\leq \tilde{\mathcal S}_{1} \| v\|_{BV(\Omega)}, \ \ \forall v\in BV(\Omega),
\end{equation}
where $\tilde{\mathcal S}_{1}= \big[(N-1)\omega_{N}^{\frac{1}{N}} \big]^{-1}$. 

\medskip 

With \eqref{soblor} at our disposal it is straightforward to carry out the proof of an existence result \bk (in the  sense of Definition \ref{def_fnon}) similar to Theorem~\ref{teo_exnon} by following the same steps of  Section~\ref{sec_fnon} with straightforward modifications\bk. It is worth mentioning that the existence results from~\cite{DGOP}  used in the proof of Theorem \ref{teo_ex} also hold for $f\in L^{N,\infty}(\Omega)$, see~\cite[Section~7]{DGOP}. Also recalling Theorem~\ref{teo_ex} (for uniqueness),  summarizing one has the following. 

\begin{theorem} 
	Let $h$ satisfy \eqref{eq:hyp_h} and let $0\le f\in L^{N,\infty}(\Omega)$ be such that $\displaystyle ||f||_{L^{N,\infty}(\Omega)}<(\tilde{\mathcal{S}}_1 h(\infty))^{-1}$. Then there exists a solution $u\in BV(\Omega)\cap L^\infty(\Omega)$ to problem \eqref{eq:PbMain} in the  sense of Definition \ref{def_fnon}\bk.   Moreover, if $f>0$ and  $h$ is decreasing,  then $u$ is the unique solution of \eqref{eq:PbMain}  in $BV(\Omega)$.
\end{theorem}

\subsection{$L^m(\Omega)$ data in the model case, $1\leq m<N$} \label{sec5a}
Here we aim to investigate  the effect that the behaviour of the nonlinearity $h$ at infinity produces on the (existence and) regularity of the solutions in case of less regular data, namely  $f\in L^m(\Omega)$ for some $m\geq1$. In order  to better present the key idea, we first consider the model case
\begin{equation}
	\label{eq:PbL1}
	\begin{cases}
		\dis -\Delta_1 u = \frac{f}{u^\gamma} & \text{in}\;\Omega,\\
		u=0 & \text{on}\;\partial\Omega,
	\end{cases}
\end{equation}
with $\gamma>0$ and $f$ a positive function  belonging to $L^m(\Omega)$ with $1\leq m<N$. 
The general case is discussed in Section~\ref{sec5b} below.
\medskip

As pointed out in~\cite{lops},  when $f$ is merely integrable, a priori one can not even prove that the solution is bounded, neither in  $BV_\mathrm{loc}(\Omega)$. Then, one has to slightly re-adapt Definition~\ref{def_fpos}.

\begin{defin}\label{def_f_L1}
	Let $0< f\in L^1(\Omega)$. A nonnegative function $u$ with $T_k(u)\in BV_{\rm loc}(\Omega)$ for every $k>0$ is a solution to problem~\eqref{eq:PbL1} if there exists $z\in \mathcal{D}\mathcal{M}^\infty(\Omega)$ with $\|z\|_{L^\infty(\Omega)^N}\le 1$ such that
	\begin{equation*} 
		-\operatorname{div}z = \frac{f}{u^\gamma} \ \ \text{as measures in }\Omega,
	\end{equation*}
	\begin{equation*} 
 		(z,DT_k(u))=|DT_k(u)| \ \ \text{as measures in } \Omega  \text{ (for every } k>0),
	\end{equation*}        
	and one of the following conditions holds:
    \begin{equation}  \label{def:bordo_L1} 
    \lim_{\varepsilon\to 0^+}\fint_{\Omega\cap B(x,\epsilon)} T_k(u (y)) \, dy = 0 \ \ \ \text{or} \ \ \ [z,\nu] (x)= -1 \ \ \ \text{for  $\mathcal{H}^{N-1}$-a.e. } x \in \partial\Omega.		
	\end{equation}
\end{defin}

\begin{remark}\label{rem:bordo_L1}
    The first condition in \eqref{def:bordo_L1} in case  $T_k(u) \in BV(\Omega)$    implies that $T_k(u)=0$ \  $\mathcal{H}^{N-1}$-$\mathrm{a.e.}$ on $\partial\Omega$. Hence\bk, in this case, one has that \eqref{def:bordo_L1} is equivalent to $ T_k(u) [z,\nu]=-T_k(u)$ \ $\mathcal{H}^{N-1}$-$\mathrm{a.e.}$ on $\partial\Omega$. Also observe that if $f\in L^N (\Omega)$ then the solution found in Theorem~\ref{teo_reg} it is also a solution in the sense of  Definition \ref{def_f_L1}.
    \triang
\end{remark}

The following theorem successfully extends the results of~\cite{bo} (see also~\cite{DCA}) to the case of the 1-Laplacian operator. Broadly speaking, the idea is to divide the solution into two parts which are studied separately. As we will see, the proof gives relevant information about the effect that the nonlinearity has on the regularity of the solution.

\begin{theorem}\label{teo_ex_L1}
	Let $0<f\in L^m(\Omega)$ with $m\geq 1$. Then there exists a unique positive solution $u\in BV(\Omega)$ to \eqref{eq:PbL1} if one of the following cases occurs:
 \begin{enumerate}[i)]
     \item $\gamma<1$ and $m=\left( \frac{1^*}{1-\gamma} \right)'=\frac{N}{(N-1)\gamma + 1}$,
     \\[-2mm]
    \item $\gamma\geq 1$ and $m=1$.
 \end{enumerate}
\end{theorem}

\begin{remark}
    When $\gamma<1$ the result is sharp, i.e., if $m<\frac{N}{(N-1)\gamma + 1}$ the solution is not expected to be in $BV(\Omega)$ (see Example~\ref{ex:sharp}). Nevertheless, recall that these solutions always satisfy that $T_k(u)\in BV(\Omega)$ for every $k>0$ (see~\cite{lops}). \triang
\end{remark}

\begin{remark}
Regarding the  first \bk item of Theorem~\ref{teo_ex_L1}, observe that there is a formal continuity in the summability exponent  of $f$, in the sense that $\frac{N}{(N-1)\gamma + 1} \to 1^+$ as $\gamma\to 1^-$ and that $\frac{N}{(N-1)\gamma + 1} \to N^-$ as $\gamma\to 0^+$. \triang
\end{remark}

\begin{proof}[Proof of Theorem \ref{teo_ex_L1}]
\mbox{}
We start recalling some preparatory  arguments that will be useful for proving both points $i)$ and $ii)$. In~\cite[Section~3]{lops}, the authors show that there exists a unique solution $u$ to problem~\eqref{eq:PbL1} in the sense of Definition~\ref{def_f_L1} such that $T_k(u)^{\sigma} \in BV(\Omega)$ for every $k>0$,  with $\sigma$ defined in~\eqref{sigma}. \bk This solution $u$ is almost everywhere finite and verifies $u>0$ $\mathrm{a.e.}$ in $\Omega$. We denote by $z$ the vector field associated to $u$.

Our strategy consists in proving that, for some $k>0$, both $T_k(u)$ and $G_k(u)$ are in $BV(\Omega)$. As $u= T_k(u) + G_k(u)$, this will imply that $u\in BV(\Omega)$.

\medskip
\textbf{Case $i)$.} 
Here we already know that $T_k(u)\in BV(\Omega)$ for every $k>0$. In what follows, we show for every fixed $k>0$ that $G_k(u)\in BV(\Omega)$.

For $\ell>k>0$, we take $G_k(T_\ell(u))$ as test function in~\eqref{eq:PbL1} obtaining, after applying \eqref{green} and \eqref{des2}, that
\begin{equation*}
    \int_\Omega (z,DG_k(T_\ell(u))) - \int_{\partial\Omega} G_k(T_\ell(u)) [z,\nu]\, d\mathcal H^{N-1} = \int_\Omega \frac{f}{u^\gamma} G_k(T_\ell(u)).
\end{equation*}

Using that $(z,DG_k(T_\ell(u))) = |DG_k(T_\ell(u))|$ as measures in $\Omega$ by Lemma~\ref{lem:Composition} and reasoning as in  Remark~\ref{rem:bordo_L1} one readily  gets 
\begin{equation*}
    \|G_k(T_\ell(u))\|_{BV(\Omega)} = \int_\Omega |DG_k(T_\ell(u))| + \int_{\partial\Omega} G_k(T_\ell(u))\, d\mathcal H^{N-1} = \int_\Omega \frac{f}{u^\gamma} G_k(T_\ell(u)).
\end{equation*}

On the right-hand side, we  apply the \bk H\"{o}lder and the Sobolev inequalities to get 
\begin{align*}
\|G_k(T_\ell(u))\|_{BV(\Omega)} &= \int_\Omega \frac{f}{u^\gamma} G_k(T_\ell(u)) \leq \int_\Omega f G_k(T_\ell(u))^{1-\gamma} \\
     &\leq \|f\|_{L^m(\Omega)} \left(\int_\Omega G_k(T_\ell(u))^{m'(1-\gamma)} \right)^\frac{1}{m'} \leq \mathcal{S}_1 \|f\|_{L^m(\Omega)} \|G_k(T_\ell(u))\|_{BV(\Omega)}^{\frac{1^*}{m'}},
\end{align*}
where in the last inequality we have used that $m'(1-\gamma) = 1^*$. Now, observe that $\frac{1^*}{m'} = 1-\gamma$. Then we have
\begin{equation*}
    \|G_k(T_\ell(u))\|_{BV(\Omega)} \leq \left( \mathcal{S}_1 \|f\|_{L^m(\Omega)} \right)^\frac{1}{\gamma},\ \forall \ell>0.
\end{equation*}

Therefore, we obtain that $G_k(T_\ell(u))$ is bounded in $BV(\Omega)$ uniformly with respect to $\ell$. Then, up to subsequences it has an $\mathrm{a.e.}$ limit belonging to $BV(\Omega)$. Since $u$ is almost everywhere finite, that limit must be $G_k(u)$ and we deduce that $G_k(u)\in BV(\Omega)$.

\medskip 
\textbf{Case $ii)$.}
We prove that $T_k(u)\in BV(\Omega)$ for any fixed $k>0$. \bk In this case, we will use the fact that we already know that $T_k(u)^\gamma\in BV(\Omega)$. Observe that function $s\mapsto \left(s+\frac{1}{n}\right)^\frac{1}{\gamma}$ is Lipschitz continuous, so one has that $v_{k,n}=\left(T_k(u)^\gamma + \frac{1}{n} \right)^\frac{1}{\gamma} - k-1 \in BV(\Omega)$. Taking $v_{k,n}$ as test function in \eqref{eq:PbL1} and dropping the right-hand side (observe that $v_{k,n}\leq 0$ for any $n\geq 1$), we obtain after applying \eqref{green} and \eqref{des2} that
\[
\int_\Omega (z,Dv_{k,n}) \leq \int_{\partial\Omega} v_{k,n}[z,\nu]\, d\mathcal H^{N-1} \leq (k+ 1) \mathcal{H}^{N-1}(\partial\Omega).
\]

As $(z,Dv_{k,n}) = |Dv_{k,n}|$ by Lemma \ref{lem:Composition}, we get
\[
\int_\Omega |Dv_{k,n}| \leq (k + 1)\mathcal{H}^{N-1}(\partial\Omega).
\]
By applying the Lebesgue Theorem, one has that $v_{k,n}\to T_k(u)-k-1$ in $L^1(\Omega)$ as $n\to\infty$. Using the lower semicontinuity, we deduce
\begin{equation*}
	\int_\Omega |DT_k(u)| = \int_\Omega |D(T_k(u)-k-1)| \leq (k + 1)\mathcal{H}^{N-1}(\partial\Omega).
\end{equation*}
In this way, since $T_k(u)$ also belongs to $L^\infty(\Omega)$, we can conclude that $T_k(u)\in BV(\Omega)$ for all $k>0$.

Now let us show that $G_k(u)\in BV(\Omega)$ for every $k>0$.  We consider $\ell >0$ and we take $G_k(T_\ell(u))$ \bk as test function in~\eqref{eq:PbL1} and arguing as before, we arrive to
\begin{align*}
\|G_k(T_\ell(u))\|_{BV(\Omega)} &= \int_\Omega \frac{f}{u^\gamma} G_k(T_\ell(u)) = \int_{\{ u \geq k\}} \frac{f}{u^\gamma} G_k(T_\ell(u)) \leq k^{1-\gamma} \int_{\{u\geq k\}} f \leq k^{1-\gamma} \|f\|_{L^m(\Omega)}, \ \forall \ell>0,
\end{align*}
where we have taken into account that $G_k(T_\ell(u))\leq u$. In this way, $G_k(T_\ell(u))$ is uniformly bounded in $BV(\Omega)$ with respect to $\ell$ and, as we have seen above, this implies $G_k(u)\in BV(\Omega)$. This concludes the proof.

\end{proof}

\subsection{The case of a generic nonlinear term}\label{sec6}\label{sec5b}

From the proof of Theorem \ref{teo_ex_L1} in Section \ref{sec5a} one could guess that only the behaviour of $h(s)$ at infinity really matters  in order to obtain \bk  solutions in $BV(\Omega)$. In fact, since essentially from $T_k(u)^{\sigma} \in BV(\Omega)$ one always can prove $T_k(u) \in BV(\Omega)$, then only $G_k(u)$ really  plays a crucial role.
\medskip

Here  we show that this is the case and that, in fact, all the previous results are still true without assuming any behaviour of $h(s)$ at zero. In order to be complete we deal with the case of a generic  nonlinear lower order term; namely, we consider problem\bk 
\begin{equation}
	\label{eq:PbF}
	\begin{cases}
		\dis -\Delta_1 u = F(x,u) & \text{in}\;\Omega,\\
		u=0 & \text{on}\;\partial\Omega,
	\end{cases}
\end{equation}
and we assume that $F(x,s)$ is a nonnegative Carath\'{e}odory function such that
\begin{equation} \label{eq:hyp_F}
    F(x,s) \leq h(s) f(x),\ \forall (x,s) \in \Omega\times [0,\infty),
\end{equation}
where $h\colon [0,\infty) \to [0,\infty]$ is a continuous function with 
\[
h(\infty):=\limsup_{s\to\infty}h(s)<  \infty,
\]
and $0\leq f\in L^m(\Omega)$ with $m\geq 1$. Without further growth restrictions, we  assume that 
\[
 h(0)= \infty.
\]
Moreover, for simplicity we suppose that 
\begin{equation} \label{eq:hyp_Fi}
F(x,0) = \infty\ \  \text{for}\ \  \mathrm{a.e.} \ x \ \ \text{in}\ \  \Omega.
\end{equation} 
This last request is only technical as it allows to obtain positive solutions and to  deal with the easiest (and clearer) definition of solution. A similar argument to the one devised in Section~\ref{sec_fnon} can be applied to address general $F(x,s)$; in this case, the approximated functions to consider would be  $F_n(x,s) = F(x,s)+\frac{1}{n}h(s)$.

\begin{defin}\label{def_f_F}
 A positive function $u$ with $T_k(u)\in BV_{\rm loc}(\Omega)$ for every $k>0$ is a solution to problem~\eqref{eq:PbF} if there exists $z\in \mathcal{D}\mathcal{M}^\infty(\Omega)$ with $\|z\|_{L^\infty(\Omega)^N}\le 1$ such that
	\begin{equation*}        
		-\operatorname{div}z = F(x,u) \ \ \text{as measures in }\Omega,
	\end{equation*}
	\begin{equation*}         
 		(z,DT_k(u))=|DT_k(u)| \ \ \text{as measures in } \Omega \text{ (for every } k>0),
	\end{equation*}        
	and one of the following conditions holds:
    \begin{equation*}  
    \lim_{\varepsilon\to 0^+}\fint_{\Omega\cap B(x,\epsilon)} T_k(u (y)) \, dy = 0 \ \ \ \text{or} \ \ \ [z,\nu] (x)= -1 \ \ \ \text{for  $\mathcal{H}^{N-1}$-a.e. } x \in \partial\Omega.		
	\end{equation*}
\end{defin}

\begin{remark}
 As we said, \bk the generality of the term $F(x,s)$ makes it difficult to encompass the non-singular case with some condition on $F(x,s)$ and to formulate, for this case, an appropriate definition of solution. The situation is clearer when considering $F(x,s) = h(s)f(x)$. In this case, $F(x,0) = \infty$ is equivalent to $f>0$ and $h(0)=\infty$. When $f\geq 0$ and $h(0)<\infty$ (i.e. in the non-singular case), the notion of solution  is  that \bk  of Definition~\ref{def_f_F}.  In any case, we underline that the following result is new even for non-singular problems. \triang
\end{remark}

We stress again  that no behaviour has been imposed on $h(s)$ at $s=0$. The following  result highlights indeed that only the behaviour of the nonlinearity at infinity plays a role in the regularity of the solution,  the behaviour at zero being negligible. This fact represents a striking  difference with respect to the $p$-Laplacian case.

\begin{theorem} \label{teo_ex_F}
Assume that $F(x,s)$ verifies~\eqref{eq:hyp_F}  and \eqref{eq:hyp_Fi}, with $0< f\in L^m(\Omega)$ for $m\geq 1$. Then, a solution $u\in BV(\Omega)$ of~\eqref{eq:PbF} in the sense of Definition \ref{def_f_F} exists if one of the following cases occurs:  
\begin{enumerate}[i)]
    \item $m= N$ and $\|f\|_{L^N(\Omega)} <  (\mathcal{S}_1 h(\infty))^{-1}$.
    \item $m=\left( \frac{1^*}{1-\theta} \right)'=\frac{N}{(N-1)\theta + 1}$ and $\limsup_{s\to\infty} h(s) s^\theta < \infty$ for some $0< \theta < 1$.\\[-2.2mm]
    \item $m=1$ and $\limsup_{s\to\infty} h(s) s < \infty$.
\end{enumerate}
Moreover, in the first case $u\in L^\infty(\Omega)$.
\end{theorem}

To prove this result, it is good to have some smoothness properties over $h(s)$. This can be done without loosing  generality. Indeed, for a given $h(s)$ as before it is  always possible to  construct (see for instance~\cite[Section~2.4]{lops}) a decreasing function $\tilde h\colon [0,\infty) \to (0,\infty]$ such that $h(s)\leq \tilde h(s)$ for every $s>0$ and verifying 
\begin{enumerate}[$a)$]
    \item $\tilde h\in C^1((0,\infty))$ and $\frac{1}{\tilde h} \in C^1([0,\infty))$ with $\left(\frac{1}{\tilde h} \right) (0) := 0$,
    \item $\left(\frac{1}{\tilde h}\right)^{-1} \colon \left[0, \frac{1}{h(\infty)} \right)\to [0,\infty)$ is locally Lipschitz in $\left(0, \frac{1}{h(\infty)} \right)$ (here symbol $^{-1}$ stands for the inverse function).
\end{enumerate}

Moreover, $\tilde h(s)$ can be chosen with the same behaviour at infinity as $h(s)$, in the sense that 
\[
\limsup_{s\to\infty} \tilde h(s) s^\theta = \limsup_{s\to\infty} h(s) s^\theta,\ \forall \theta\geq 0.
\]

Then, there is no restriction in assuming, in \eqref{eq:hyp_F}, that $h(s)$ is decreasing and verifies all the smoothness properties above-mentioned. For the sake of clarity, we denote $\Gamma(s):= \left(\frac{1}{h}\right) (s)$ for every $s\geq 0$.  Observe that $\Gamma(s)$ is increasing and that $\Gamma^{-1}(s)$ is locally Lipschitz in open intervals not containing 0. It is worth noting that, as a consequence of $\Gamma$ being derivable at 0, we are essentially dealing with strong singularities.
\medskip

When case $i)$ of Theorem~\ref{teo_ex_F} takes place, in~\cite[Section~8.2]{DGOP} it is proved that a solution $u\in BV_\mathrm{loc}(\Omega)$ to~\eqref{eq:PbF} such that $\Gamma(u)\in BV(\Omega)$ exists. Furthermore, one can easily follow the steps of~\cite{lops} (combined with the ideas of~\cite[Section~8.2]{DGOP}) to show that, in cases $ii)$ and $iii)$ of Theorem~\ref{teo_ex_F}, a solution $u$ of~\eqref{eq:PbF} does exist satisfying $T_k(u)\in BV_\mathrm{loc}(\Omega)$ and $\Gamma(T_k(u))\in BV(\Omega)$ for every $k>0$.
\medskip

Intuitively, the reason for having $\Gamma(T_k(u))\in BV(\Omega)$ for every $k>0$ in any case is that, if one formally takes $\Gamma(T_k(u))$ as test function in~\eqref{eq:PbF}, one can arrive to
\begin{align*}
\int_\Omega |D\Gamma(T_k(u))| + \int_{\partial\Omega} \Gamma(T_k(u))\, d\mathcal{H}^{N-1} &= \int_\Omega F(x,u) \, \Gamma(T_k(u)) \leq \int_\Omega fh(u) \, \Gamma(T_k(u))\\
&= \int_{\{u\leq k\}} f + \Gamma(k) \int_{\{u>k\}} fh(u) <\infty,
\end{align*}
where the last inequality is always true since $h(\infty)<\infty$.
\medskip

Having $\Gamma(T_k(u))\in BV(\Omega)$ is enough to deduce that $T_k(u)\in BV(\Omega)$, as the following regularity result shows.
\begin{Proposition}\label{prop_reg_F}
    Let $\Gamma\colon [0,\infty)\to [0,\infty)$ be an increasing function such that $\Gamma^{-1}(s)$ is locally Lipschitz in open intervals not containing 0. If $u$ is a solution to~\eqref{eq:PbF} such that $\Gamma(T_k(u))\in BV(\Omega)$ for some $k>0$, then $T_k(u)\in BV(\Omega)$.
\end{Proposition}

\begin{proof}
The key here is that function $s\mapsto \Gamma^{-1}\left(s+\frac{1}{n} \right)$ is Lipschitz on bounded intervals containing 0. Then, for every $n\in \mathbb{N}$ we have that $v_{k,n}= \Gamma^{-1} \left(\Gamma(T_k(u)) + \frac{1}{n} \right) - \Gamma^{-1} (\Gamma(k) + 1) \in BV(\Omega)$. Observe that since both $\Gamma$ and $\Gamma^{-1}$ are increasing, one has that $v_{k,n}\leq 0$.

We point out that $\Gamma(\infty) :=\lim_{s\to\infty} \Gamma(s)$ may be finite and then the domain of $\Gamma^{-1}$ would be only $[0,\Gamma(\infty))$. Without loss of generality, we assume that $\Gamma(k) + \frac{1}{n}$ belongs to the domain of $\Gamma^{-1}$ for every $n\in\N$ (otherwise it suffices to consider $n\geq n_0$ with $n_0$ suitably large).

Taking $v_{k,n}$ as test function in \eqref{eq:PbF} and dropping the right-hand side (recall that $v_{k,n}\leq 0$), we obtain after applying \eqref{green} and \eqref{des2} that
\[
\int_\Omega (z,Dv_{k,n}) \leq \int_{\partial\Omega} v_{k,n}[z,\nu]\, d\mathcal H^{N-1} \leq \Gamma^{-1} (\Gamma(k) + 1) \mathcal{H}^{N-1}(\partial\Omega).
\]

As $(z,Dv_{k,n}) = |Dv_{k,n}|$ by Lemma \ref{lem:Composition}, we get
\[
\int_\Omega |Dv_{k,n}| \leq \Gamma^{-1}(\Gamma(k) + 1)\mathcal{H}^{N-1}(\partial\Omega).
\]

By Lebesgue Theorem, one has that $v_{k,n}\to T_k(u)-\Gamma^{-1} (\Gamma(k) + 1 )$ in $L^1(\Omega)$ as $n\to\infty$. Using the lower semicontinuity, we deduce
\begin{equation*}
\int_\Omega |DT_k(u)| = \int_\Omega \left| D(T_k(u)-\Gamma^{-1} (\Gamma(k) + 1)) \right| 
\leq \Gamma^{-1} \left(\Gamma(k) + 1 \right) \mathcal{H}^{N-1}(\partial\Omega).
\end{equation*}
In this way, since $T_k(u)$ also belongs to $L^\infty(\Omega)$, we can conclude that $T_k(u)\in BV(\Omega)$.
\end{proof}

The above proposition shows that one can always find solutions to \eqref{eq:PbF} having any of their truncations with finite energy. Then, in order to show that $u\in BV(\Omega)$, one just needs to take care about $u$ at infinity; in other words, only $G_k(u)$ plays a key role in proving that $u$ has finite energy. Therefore, to prove Theorem~\ref{teo_ex_F}, we focus our attention on $G_k(u)$.
\bk
\begin{proof}[Proof of Theorem~\ref{teo_ex_F}]
In any case, we have seen that there exists a solution $u$ of~\eqref{eq:PbF} in the sense of Definition~\ref{def_f_F} such that $\Gamma(T_k(u))\in BV(\Omega)$ for every $k>0$. Then, Proposition~\ref{prop_reg_F} applies and one has   $T_k(u)\in BV(\Omega)$ for every $k>0$. Therefore, since $u=T_k(u)+G_k(u)$, it suffices to show that $G_k(u)\in BV(\Omega)$.

\medskip 
\textbf{Case $i)$.}
\mbox{} In this case\bk, one can prove that $\|G_k(u)\|_{BV(\Omega)} = 0$ for some $k$ large. This implies that $u\in L^\infty(\Omega)$ and that $u\in BV(\Omega)$. The proof is almost identical to the one of Theorem~\ref{teo_reg} and we omit it.

\medskip

\textbf{Case $ii)$.}
 For $\ell>k>0$, we take $G_k(T_{\ell}(u)) \in BV(\Omega)\cap L^\infty(\Omega)$ as test function in~\eqref{eq:PbF} and,  after applying~\eqref{green},~\eqref{des2} and~\eqref{eq:hyp_F}, one gets  that
\begin{equation*}
    \int_\Omega (z,DG_k(T_\ell(u))) - \int_{\partial\Omega} G_k(T_\ell(u)) [z,\nu]\, d\mathcal H^{N-1} = \int_\Omega F(x,u) G_k(T_\ell(u)) \leq \int_{\{u\geq k\}} f h(u) G_k(T_\ell(u)).
\end{equation*}
Using that $(z,DG_k(T_\ell(u)))=|DG_k(T_\ell(u))|$ as measures in $\Omega$ by Lemma~\ref{lem:Composition} and arguing as in Remark~\ref{rem:bordo_L1}, on the left-hand side we obtain
\begin{equation}
    \label{eq:Pf_ex_F_1}
    \|G_k(T_\ell(u))\|_{BV(\Omega)} = \int_\Omega |DG_k(T_\ell(u))| + \int_{\partial\Omega} |G_k(T_\ell(u))|\, d\mathcal H^{N-1} \leq \int_{\{u\geq k\}} f h(u) G_k(T_\ell(u)).
\end{equation}

By hypothesis, we know that $\limsup_{s\to\infty} h(s)s^\theta<\infty$ for some $0<\theta <1$. This implies the existence of some $c_1>0$ and some $\tilde k>0$ such that
\begin{equation}
    \label{eq:Pf_ex_F_2}
    h(s) \leq \frac{c_1}{s^\theta},\ \forall s\geq \tilde k.
\end{equation}

From now on, we fix $k=\tilde k$. Then, observing that $G_k(T_\ell(s))\leq s$, from~\eqref{eq:Pf_ex_F_1} and~\eqref{eq:Pf_ex_F_2} we deduce
\begin{equation}
    \label{eq:Pf_ex_F_3}
    \|G_k(T_\ell(u))\|_{BV(\Omega)} \leq  c_1 \bk \int_{\{u\geq k\}} f G_k(T_\ell(u))^{1-\theta}.
\end{equation}

Using H\"{o}lder and Sobolev inequalities in the previous we get
\begin{align}
\label{eq:Pf_ex_F_4}
     \|G_k(T_\ell(u))\|_{BV(\Omega)} \leq  c_1 \bk \|f\|_{L^m(\Omega)} \left(\int_\Omega G_k(T_\ell(u))^{m'(1-\theta)} \right)^\frac{1}{m'} \leq  c_1 \bk\mathcal{S}_1 \|f\|_{L^m(\Omega)} \|G_k(T_\ell(u))\|_{BV(\Omega)}^{\frac{1^*}{m'}},
\end{align}
where the last inequality is true since $m'(1-\theta) = 1^*$. As $\frac{1^*}{m'} = 1-\theta$, we have
\begin{equation*}
\|G_k(T_\ell(u))\|_{BV(\Omega)} \leq ( c_1 \bk\mathcal{S}_1 \|f\|_{L^m(\Omega)})^\frac{1}{\theta}.
\end{equation*}

Now we  pass to the limit in $\ell\to \infty$.  The previous implies that \bk $G_k(T_\ell(u))$ is bounded in $L^1(\Omega)$ and then Fatou Lemma gives $G_k(u)\in L^1(\Omega)$. In this way, Lebesgue Theorem can be applied in order to get $G_k(T_\ell(u))\to G_k(u)$ in $L^1(\Omega)$ as $\ell\to  \infty$. Therefore, by lower semicontinuity one can pass to the limit in~\eqref{eq:Pf_ex_F_4} to conclude
\begin{equation*}
\|G_k(u)\|_{BV(\Omega)} \leq ( c_1 \bk \mathcal{S}_1 \|f\|_{L^m(\Omega)})^\frac{1}{\theta}
\end{equation*}
and thus $G_k(u)\in BV(\Omega)$.
\mbox{}

\medskip 
\textbf{Case $iii)$.}
Following the same lines as in the previous case (but now using that $\limsup_{s\to\infty} h(s)s<\infty$), we can deduce, in place of~\eqref{eq:Pf_ex_F_3}, that  
\[
\|G_k(T_\ell(u))\|_{BV(\Omega)} \leq C\int_{\{u\geq k\}} f \leq C\|f\|_{L^1(\Omega)}.
\]
Therefore, one can pass to the limit as in the previous case to obtain that $G_k(u)\in BV(\Omega)$.
\end{proof}

\section{Further  comments, extensions and explicit  examples}
\label{sec_final}

In this section we provide some final comments together with the discussion of some open problems and   possible extensions, as well as some  explicit examples. 

\subsection{Some open questions} 

As already explained in the introduction, when dealing with the $p$-Laplace operator, the $\gamma$-threshold for having finite energy solutions in problem~\eqref{op} reads as
\begin{equation}
\label{eq:threshold_1}
    \gamma < \frac{2p-1}{p-1},
\end{equation}
at least when $f$ is smooth  in $\overline{\Omega}$ \bk and bounded away from zero. This was proved first in~\cite{lm} for the case $p=2$, for general $p>1$ in~\cite{S} and, in this paper, we have completed the study of this phenomenon proving that it also holds for $p=1$.
\medskip

On the other hand, in~\cite{OP} the authors consider the Laplace operator (i.e. $p=2$) and a  positive \bk $f$ only  belonging to $L^m(\Omega)$, with $m> 1$. For $\gamma>1$, they show that the unique solution for problem~\eqref{pbintro} belongs to $H_0^1(\Omega)$ for any such $f$'s if and only if
\begin{equation}
\label{eq:threshold_2}
\gamma < 3-\frac{2}{m}.
\end{equation}

Up to our knowledge, this threshold have not been extended to the $p$-Laplacian case. The following example, mainly adapted from~\cite[Example $2$]{OP}, allows us to conjecture what should be this range for the $p$-Laplacian and to guess its behaviour when $p$ tends to $1^{+}$.

\begin{example}\label{ex1}
Let $p>1$ and let $\gamma>1$. For $\eta>0$, define function $u = (1-|x|^{2})^{\eta}$ and define $f(x)=u^\gamma(-\Delta_p u)$. Then, $u$ is a solution to
\begin{equation*}
    \begin{cases}
        \displaystyle -\Delta_p u= \frac{f(x)}{u^{\gamma}}   &\text{in }\ B_{1}(0), \\
        u=0 & \text{on }\ \partial B_{1}(0), 
    \end{cases}
\end{equation*}
where $B_R(0)$ is the ball centered at the origin with radius $R>0$, and
$$
f(x)\sim \frac{1}{(1-|x|^{2})^{p-\eta(\gamma+p-1)}}.
$$ 
For fixed  $m\geq 1$, one can require
$$
\eta >\frac{p-\frac{1}{m}}{\gamma+p-1}
$$
in order to get $f\in L^m(\Omega)$. With regard to $u$, one always has that $u^{\frac{\gamma-1+p}{p}}\in W^{1,p}_0(\Omega)$ but $u\in W^{1,p}_0(\Omega)$ only when $\eta>\frac{p-1}{p}$. Requiring that
$$
\frac{p-\frac{1}{m}}{\gamma+p-1} > \frac{p-1}{p},
$$
one yields to
\begin{equation}
\label{eq:threshold_3}
\gamma < \frac{2p-1}{p-1} - \left(\frac{p}{p-1}\right) \frac{1}{m} = \frac{\left(2-\frac{1}{m}\right)p -1}{p-1}.
\end{equation}
Therefore, $u\in W^{1,p}_0(\Omega)$ when $\gamma$ verifies~\eqref{eq:threshold_3}.
\triang
\end{example}

Confronting~\eqref{eq:threshold_3} with~\eqref{eq:threshold_1} and~\eqref{eq:threshold_2}, we can see some continuity in the parameters. We conjecture that, for $\gamma>1$,  threshold~\eqref{eq:threshold_3} should be optimal in order to obtain $W_0^{1,p}(\Omega)$ solutions of problem~\eqref{op} for  positive \bk data $f\in L^m(\Omega)$, $m>1$. 
\medskip

The results proven in this work give strength to the conjecture by showing that, apart from $p=2$, it also holds true for $p=1$ (see Theorem~\ref{teo_ex_L1}). When $p=1$, threshold~\eqref{eq:threshold_3} becomes unbounded and, in this case, the conjecture says that a solution of~\eqref{op} exists for every $\gamma>1$ when $m>1$.  In fact, we have shown that this is the case even when $m=1$.
\medskip

To conclude, such examples also show that requiring a power of the solution to belong to the energy space, as Theorem~\ref{teo_reg} does, is quite natural for such  problems.

\subsection{Sharpness of Theorem \ref{teo_ex_L1}}
We present here an example taken from~\cite{lops} which shows that when $\gamma<1$ and $f\in L^m(\Omega)$ with $m<\frac{N}{(N-1)\gamma + 1}$ solutions to~\eqref{eq:PbL1} are not expected to be in $BV(\Omega)$.

\begin{example} \label{ex:sharp}
We look for radial solutions, so we consider again $\Omega$ as the ball $B_R(0)$ centered at the origin with radius $R>0$. We take the positive datum $f(x)= \frac{N-1}{|x|^q}$ with $1<q<N$. Observe that $f\in L^m(\Omega)$ if and only if $m<\frac{N}{q}$. Then, we study problem
\begin{equation*}
    \begin{cases}
        \displaystyle -\operatorname{div}\left( \frac{Du}{|Du|}\right)= \frac{N-1}{|x|^q}\frac{1}{u^{\gamma}}   &\text{in }\ B_{R}(0), \\
        u\geq 0 & \text{in }\ B_{R}(0), \\
        u=0 & \text{on }\ \partial B_{R}(0).
    \end{cases}
\end{equation*}

Assume that the solution $u(x)$ is radially decreasing, i.e., $u(x)=g(|x|)$ with $g'(s)<0$ for $0\leq s\leq R$. In this case, one always has that $z(x) = \frac{Du}{|Du|} = -\frac{x}{|x|}$ and $-\operatorname{div} z = \frac{N-1}{|x|}$. Moreover, the boundary condition is always satisfied in the sense $[z,\nu](x)=-1$ for every $x\in\partial B_R(0)$. Then, equation becomes
\[
\frac{N-1}{|x|} = \operatorname{div} z = \frac{N-1}{|x|^q}\frac{1}{g(|x|)^{\gamma}}
\]
and one obtains that $u(x)=g(|x|) = |x|^\frac{1-q}{\gamma}$. We point out that $u$ is positive and unbounded. This highlights the fact that, in general, one cannot expect bounded solutions to~\eqref{eq:PbL1}.
\medskip

On the other hand, observe that $u$ belongs to $BV(B_R(0))$ if and only if $\gamma>\frac{q-1}{N-1}$. We stress that this condition is always satisfied when $\gamma\geq 1$, while when $\gamma<1$ is only verified when $q< (N-1)\gamma + 1$. Then, when $\gamma<1$ the solution is in $BV(B_R(0))$ if and only if $\frac{N}{(N-1)\gamma + 1} < \frac{N}{q}$, i.e. if $f\in L^m(\Omega)$ with $m=\frac{N}{(N-1)\gamma + 1}$.
 Finally, notice that in any case the truncations are in $BV(B_R(0))$ (to be compared with Definition~\ref{def_f_L1}). 
\triang
\end{example}

\subsection{Non-constant solutions reaching the zero boundary datum at some points of $\partial\Omega$}

The degeneracy of the solutions associated to 1-Laplacian problems makes one wonders the existence of solutions which take the zero value at least at some point on the boundary of $\Omega$. This kind of degeneracy can be observed in solutions found in Example~\ref{ex:sharp}. But the situation can be even worse. For instance, when $\Omega$ is a $C^{1,1}$ domain, the unique solution of problem
\begin{equation}
		\label{pbe}
		\begin{cases}
			\dis -\Delta_1 u = \frac{1}{u^{\gamma}} & \text{in}\;\Omega,\\
			u=0 & \text{on}\;\partial\Omega,
		\end{cases}
\end{equation}  
is  a  suitable power of the Cheeger constant, namely
\[
u(x)=\left(\frac{|\Omega|}{\mathrm{Per}(\Omega)}\right)^{\frac{1}{\gamma}}.
\]

We give here a third example, presented in {\cite[Example $2$]{gop}}, which shows the existence of non-constant solutions of~\eqref{pbe} taking the value zero at some points of $\partial \Omega$.
\bk

\begin{example}
	\label{example2} Let $\Omega$ be a convex open set and let $H_{\Omega}(x)$ denote the variational mean curvature of $\Omega$ (see \cite{BGT} for details). In \cite{MP} it is proven that $-H_{\Omega} (x)$ is a large solution to $\Delta_1 v= v$, i.e. 
	\begin{equation*} 
        \begin{cases}
			\displaystyle \Delta_1 v= v &  \text{in}\, \Omega, \\
			v=\infty & \text{on}\ \partial \Omega\,.	
        \end{cases}
	\end{equation*}
	Let us underline that it can be shown that $|| H_\Omega ||_{L^{\infty}(\rn)}< \infty$ if and only if $\Omega$ is of class $C^{1,1}$. 
	These  solutions are locally bounded assuming the (large) datum $\infty$  at non-regular points of $\Omega$ (e.g. at corners).  		
	Now observe that the change of variable $u=v^{-\frac{1}{\gamma}}$ formally leads to a non-constant solution to \eqref{pbe}. Indeed, in general $H_\Omega (x)$ is known to be non-constant  if the set is not calibrable (see \cite{BGT,acc}); for instance if $\Omega$ is not $C^{1,1} $ (say a square), then  $u=v^{-\frac{1}{\gamma}}$ is positive everywhere and it attains the value $0$ only at the corners of $\Omega$. 
	\begin{figure}[htbp]\centering
		\includegraphics[width=2in]{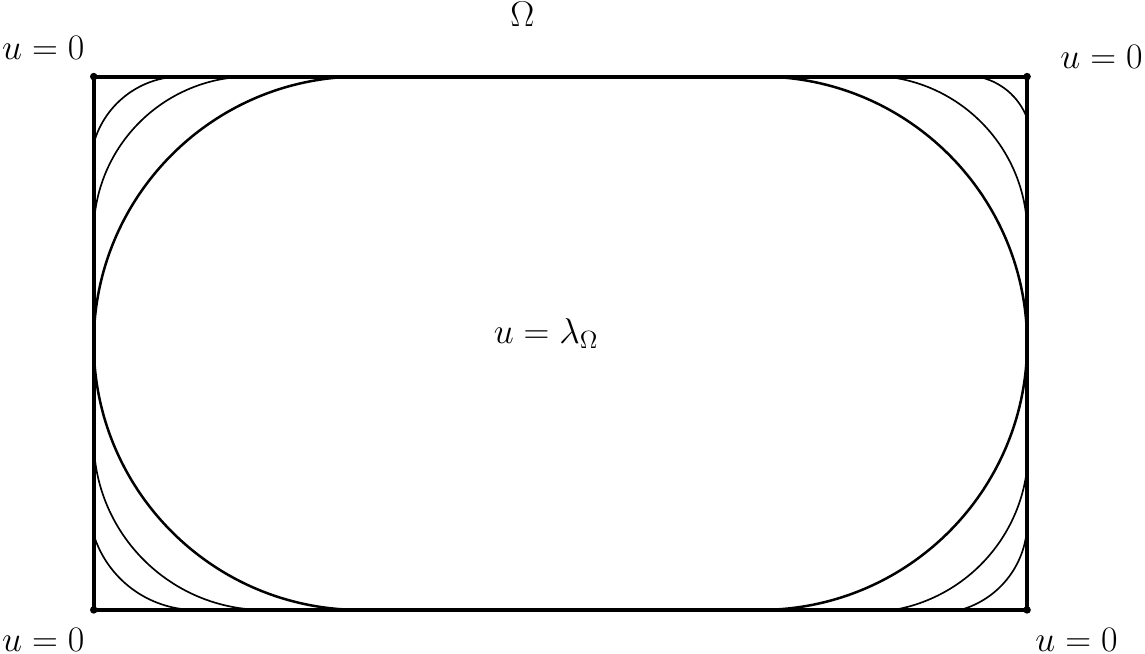}
		\caption{A non-constant solution $u= \bk (-H_\Omega(x))^{-\frac{1}{\gamma}}$ to \eqref{pbe}} \label{calo}
	\end{figure}
	We finally underline that $u$ is constant inside the (unique) Cheeger $C$ set contained in $\Omega$;  this constant $\lambda_\Omega$ is a suitable power of the Cheeger constant (see Figure \ref{calo}). \bk
 \triang
\end{example}

\subsection{Equations with first order terms}
\label{first}

Here we highlight that an analogous result to the one given in Theorem \ref{teo_exnon} holds in presence of first order reaction terms. In order to be concrete, let us have in mind nonnegative solutions to the following problem which has been dealt  with \bk in \cite{gop}
\begin{equation}\label{pbgrad}
	\begin{cases}
		\displaystyle -\Delta_1 u= g(u) |Du| + h(u)f   &\text{in}\, \ \ \Omega, \\
		u=0 & \text{on}\; \partial\Omega, 
	\end{cases}
\end{equation}
where $g$ is a nonnegative function which can blow at the origin as 
\begin{equation}\label{g}
	\displaystyle \exists\; c_2, s_2> 0\;\ \text{such that}\;\  g(s)\le \frac{c_2}{s^\theta} \ \text{ if } s\leq s_2 \text{ and } 0<\theta<1.
\end{equation}

The formal definition of solution for nonnegative $f$ can be found in~\cite[Definition~5.1]{gop}. When $f$ is positive, this definition is like Definition~\ref{def_fnon} but with the addition of hypothesis $g(u^*)\in L^1_{\rm loc}(\Omega, |Du|)$.
\medskip

In~\cite[Theorem 5.3]{gop}, the authors show the existence of a solution $u\in BV_\mathrm{loc}(\Omega)\cap L^\infty(\Omega)$ with $u^{\max(1,\gamma)} \in BV(\Omega)$. Reasoning similarly to the proof of Theorem~\ref{teo_exnon}, we can state the following existence result.

\begin{theorem} 
	Let $g$ satisfy \eqref{g}, let $h$ satisfy \eqref{eq:hyp_h},  and let $f\in L^N(\Omega)$ be nonnegative and such that $\|f\|_{L^N(\Omega)} < \left(\mathcal{S}_1h(\infty)\right)^{-1}$.	Then there exists a solution $u\in BV(\Omega)\cap L^\infty(\Omega)$ to problem \eqref{pbgrad}. 
\end{theorem}

\subsection{Equations with sign}
The result in Section \ref{first} relies on the robustness of our argument and it suggests that the regularity properties proved along the entire paper hold in more general situations. Indeed if one formally deals with equations as 
$$-\Delta_1 u \ge 0,$$
then it is always possible to show that any solution $u\in BV(\Omega)$ provided  $u^{\alpha}\in BV(\Omega)$ for some $\alpha>1$. 
\medskip

In order to be concrete, this property is true when there exists some vector $z\in \mathcal{DM}^\infty(\Omega)$ with $\|z\|_{L^\infty(\Omega)^N} \leq 1$ such that $-\operatorname{div}z \geq 0$ as measures in $\Omega$ and $(z, DT_k(u)) = |DT_k(u)|$ as measures in $\Omega$ for every $k>0$.

\section*{Acknowledgements}
The first author has been funded by Junta de Andaluc\'ia (grant FQM-194),
by the Spanish Ministry of Science and Innovation, Agencia Estatal de Investigaci\'on (AEI)
and Fondo Europeo de Desarrollo Regional (FEDER) (grant PID2021-122122NB-I00) and by
the FPU predoctoral fellowship of the Spanish Ministry of Universities (FPU21/04849).
The second and the third author  have been  partially supported by the Gruppo Nazionale per l’Analisi Matematica, la Probabilità e le loro Applicazioni (GNAMPA) of the Istituto Nazionale di Alta Matematica (INdAM).

\section*{Conflict of interest declaration}

The authors declare no competing interests.

\section*{Data availability statement}
 We do not analyse or generate any datasets, because our work
proceeds within a theoretical and mathematical approach. One
can obtain the relevant materials from the references below.

\end{document}